\documentclass[10pt,letterpaper]{article}

\usepackage{titling}
\usepackage{float}
\usepackage{tabularx}
\usepackage{authblk}

\usepackage{chngcntr}
\usepackage[utf8]{inputenc}
\usepackage{arydshln}

\usepackage{amsfonts}
\usepackage{amsthm,physics}
\usepackage{amsmath}
\usepackage{amscd}
\usepackage{t1enc}
\usepackage[mathscr]{eucal}
\usepackage{indentfirst}
\usepackage{graphicx}
\usepackage{graphics}
\usepackage{pict2e}
\usepackage{epic}
\numberwithin{equation}{section}
\usepackage[margin=2.9cm]{geometry}
\usepackage{epstopdf}
\usepackage{amsmath,amsthm,verbatim,amssymb,amsfonts,amscd,graphicx}
\usepackage{mathtools}
\usepackage[dvipsnames]{xcolor}
\usepackage{authblk}
\usepackage{bm}
\usepackage{soul}

\usepackage[dvipsnames]{xcolor}
\usepackage[margin=3cm]{caption}
\usepackage{amsmath}
\usepackage{amsthm}
\usepackage{amssymb}
\usepackage{amscd}
\usepackage{amsmath,amsthm,verbatim,amssymb,amsfonts,amscd,graphicx}
\usepackage{mathtools}
\usepackage{hyperref}
\usepackage{subfigure}
\usepackage{stmaryrd}
\usepackage{booktabs}
\usepackage{mdframed}
\usepackage{upgreek}
\usepackage{mathtools} \mathtoolsset{showonlyrefs=true} \MakeRobust{\eqref}
\usepackage{listings}
\usepackage{color}
\usepackage{fullpage}
\usepackage{stackrel}
\usepackage{textcomp}
\usepackage{enumerate}
\usepackage{accents}
\usepackage{graphicx}
\usepackage{epstopdf}
\usepackage{placeins}
\usepackage{bbm}
\usepackage{bm}
\usepackage[mathscr]{eucal}
\usepackage[linesnumbered,boxed,titlenumbered,ruled,nofillcomment]{algorithm2e}
\usepackage[T1]{fontenc}
\usepackage{lmodern}
\usepackage{scalerel}
\usepackage{wrapfig}
\usepackage{algorithmic}
\usepackage{float}
\usepackage{siunitx}

\definecolor{lbcolor}{rgb}{0.9,0.9,0.9}
\theoremstyle{plain}
\newtheorem{theorem}{Theorem} \numberwithin{theorem}{section}

\newtheorem{prop}[theorem]{Proposition}

\newtheorem{lemma}[theorem]{Lemma}

\theoremstyle{definition}
\newtheorem{remark}[theorem]{Remark}

\newtheorem{definition}[theorem]{Definition}
\newtheorem{assumption}[theorem]{Assumption}

\definecolor{mydarkblue}{rgb}{0,0.08,0.45}
\hypersetup{ %
	colorlinks=false,
	linkcolor=mydarkblue,
	citecolor=mydarkblue,
	filecolor=mydarkblue,
	urlcolor=mydarkblue}

\newcommand\Eb{\mathbb{E}}

\newcommand\Pb{\mathbb{P}}

\newcommand\Rb{\mathbb{R}}

\newcommand{\zerobf}[0]{\bm{0}}

\newcommand{\Abf}[0]{\bm{A}}

\newcommand{\ybf}[0]{\bm{y}}

\newcommand{\Jc}{\mathcal{J}}

\newcommand{\range}[0]{\operatorname{range}}
\newcommand{\orth}[0]{\operatorname{orth}} 

\DeclareMathOperator*{\argmin}{arg\,min}

\newcommand{\Uniform}{\operatorname{Uniform}}

\DeclareMathOperator*{\Ind}{\operatorname{Ind}}

\newcommand{\defeq}{\stackrel{\text{\tiny \textnormal{def}}}{=}}  

\definecolor{mypink1}{rgb}{0.858, 0.188, 0.478}
\definecolor{mypink2}{RGB}{219, 48, 122}
\definecolor{mypink3}{cmyk}{0, 0.7808, 0.4429, 0.1412}
\definecolor{mygray}{gray}{0.6}

\newlength{\leftstackrelawd}
\newlength{\leftstackrelbwd}
\def\leftstackrel#1#2{\settowidth{\leftstackrelawd}%
{${{}^{#1}}$}\settowidth{\leftstackrelbwd}{$#2$}%
\addtolength{\leftstackrelawd}{-\leftstackrelbwd}%
\leavevmode\ifthenelse{\lengthtest{\leftstackrelawd>0pt}}%
{\kern-.5\leftstackrelawd}{}\mathrel{\mathop{#2}\limits^{#1}}}

\def\R{{\mathbb R}}

\def\PP{{\mathcal P}}

\def\N{{\mathbb N}}

\def\E{{\mathbb E}}
\def\R{{\mathbb R}}

\def\P{{\mathbb P}}
\def\V{{\mathbb V}}

\def\cor{{\textup{corr}}}

\def\bfqs{{\mathsf {BFB}}}

\def\e{{\varepsilon}}

\def\u{{\mathbf{u}}}

\newcommand{\bs}[1]{\boldsymbol{#1}}

\newlength{\dhatheight}
\newcommand{\doublehat}[1]{%
    \settoheight{\dhatheight}{\ensuremath{\hat{#1}}}%
    \addtolength{\dhatheight}{-0.35ex}%
    \hat{\vphantom{\rule{1pt}{\dhatheight}}%
    \smash{\hat{#1}}}}

\hypersetup{
    colorlinks=true,
}
\usepackage[textsize=footnotesize]{todonotes}

\SetCommentSty{mycommfont}

\usepackage[backend=bibtex,firstinits=true,style=alphabetic,maxbibnames=9,natbib=true,url=false,sorting=nyt,doi=true,backref=false]{biblatex}

\renewbibmacro{in:}{\ifentrytype{article}{}{\printtext{\bibstring{in}\intitlepunct}}}

\usepackage[margin=2.9cm]{geometry}

\usepackage{xcolor}

\usepackage{booktabs}


\title{Quadrature Sampling of Parametric Models with Bi-fidelity Boosting}

\author[1]{Nuojin Cheng\thanks{Equal contribution.}}
\author[2]{Osman Asif Malik$^*$}
\author[3]{Yiming Xu$^*$}
\author[1]{Stephen Becker}
\author[4]{Alireza Doostan}
\author[5]{Akil Narayan}
\affil[1]{Department of Applied Mathematics, University of Colorado Boulder (\href{mailto:nuojin.cheng@colorado.edu}{nuojin.cheng@colorado.edu}, \href{mailto:stephen.becker@colorado.edu}{stephen.becker@colorado.edu})}
\affil[2]{Applied Mathematics \& Computational Research Division, Lawrence Berkeley National Laboratory (\href{mailto:oamalik@lbl.gov}{oamalik@lbl.gov})}

\affil[3]{Corporate Model Risk, Wells Fargo (\href{mailto:yiming.xu@wellsfargo.com}{yiming.xu@wellsfargo.com})}
\affil[4]{Smead Aerospace Engineering Sciences Department, University of Colorado Boulder (\href{mailto:alireza.doostan@colorado.edu}{alireza.doostan@colorado.edu})}
\affil[5]{Scientific Computing and Imaging Institute, and Department of Mathematics, University of Utah (\href{mailto:akil@sci.utah.edu}{akil@sci.utah.edu})}

\date{}

\begin{document}

\maketitle

\begin{abstract}

Least squares regression is a ubiquitous tool for building emulators (a.k.a.\ surrogate models) of problems across science and engineering for purposes such as design space exploration and uncertainty quantification. When the regression data are generated using an experimental design process (e.g., a quadrature grid) 
involving computationally expensive models, or when the data size is large, sketching techniques have shown promise to reduce the cost of the construction of the regression model while ensuring accuracy comparable to that of the full data. However, random sketching strategies, such as those based on leverage scores, lead to regression errors that are random and may exhibit large variability. To mitigate this issue, we present a novel boosting approach that leverages cheaper, lower-fidelity data of the problem at hand to identify the {\it best} sketch among a set of candidate sketches. This in turn specifies the sketch of the intended high-fidelity model and the associated data. We provide theoretical analyses of this bi-fidelity boosting (BFB) approach and discuss the conditions the low- and high-fidelity data must satisfy for a successful boosting. In doing so, we derive a bound on the residual norm of the BFB sketched solution relating it to its ideal, but computationally expensive, high-fidelity boosted counterpart. Empirical results on both manufactured and PDE data corroborate the theoretical analyses and illustrate the efficacy of the BFB solution in reducing the regression error, as compared to the non-boosted solution. 

\end{abstract}

\section{Introduction}\label{sec:intro}

Computational models are becoming central tools in analysis, design, and prediction. In these models, input parameters are often modeled as a random vector $\bs{p}$ to account for either uncertainty in precise values of these parameters, or as a means to model variability of parameters in order to assess robustness of an output \cite{le2010spectral,smith2013uncertainty}. We consider such types of models given a (possibly non-linear) parameter-to-output map,
\begin{align*}
  b &= \mathcal{T}(\bs{p}), & \mathcal{T}: \Rb^q \rightarrow \R.
\end{align*}
A canonical example is when $\mathcal{T}$ is a measurement functional (e.g., the spatial average) operating on the solution to an elliptic partial differential equation (PDE) whose formulation contains random variables that, e.g., parameterize the diffusion coefficient. Hence, $\mathcal{T}$ is the composition of a measurement functional with the solution map of a parametric PDE. By placing a probability distribution on $\bs{p}$ that reflects a model of uncertainty, the goal of forward uncertainty quantification (UQ) is to quantify the resulting randomness in $b(\bs{p})$, 
frequently via statistics.
Since explicit formulas revealing the dependence of $b$ on $\bs{p}$ are typically not available, one resorts to approximations. One such
sampling-based 
approach that we focus on is that of polynomial chaos (PC) methods \cite{ghanem2003stochastic,xiuwiener--askey2002} using variants of stochastic collocation \cite{xiuhigh-order2005}. 

In this paper we consider building emulators for forward UQ via a non-intrusive least squares-based PC strategy. More precisely, we assume an \textit{a priori} form for an emulator $b_V$:
\begin{align}
\label{eq:function-approximation}
  b(\bs{p}) \approx b_V(\bs{p}) &\coloneqq \sum_{j=1}^d x^\ast_j \psi_j(\bs{p}), & V &\coloneqq \mathrm{span}\{\psi_1, \ldots, \psi_d\},
\end{align}
where $\psi_j$ are fixed, known functions (in PC approaches they are multivariate polynomial functions of $\bs{p}$), and the coefficients $x^\ast_j$ must be determined. We identify these coefficients through data collected from evaluating $b$ on a prescribed quadrature rule $\{(\bs{p}_n, w_n)\}_{n =1}^N$, with quadrature nodes $\bs{p}_n$ and positive weights $w_n$. The coefficients $x^\ast_j$ are then chosen as the solution to a quadrature-based least squares problem,
\begin{align}\label{myleastsquares}
  \bs{x}^\ast &= \argmin_{\bs{x} \in \R^d} \| \bs{A} \bs{x} - \bs{b} \|_2^2, & \bm A(n,j) &= \sqrt{w}_n \psi_j(\bs{p}_n), & \bm b(n) &= \sqrt{w}_n b(\bs{p}_n),
\end{align}
where $\bs{A}\in\R^{N\times d}$ is referred to as the \emph{design matrix} of the problem. 
Once $\bs{x}^\ast$ is computed, the emulator $b_V$ is easily manipulated and computationally analyzed to compute (approximate) statistics for $b$ or the sensitivity of $b$ to each entry of $\bs{p}$. The challenge with this approach is that when $\dim \bs{p} = q \gg 1$, then designing an appropriately accurate quadrature rule requires $N \gg 1$ samples of $b$, which is prohibitively expensive when such evaluations amount to PDE solutions. (For example a $q$-dimensional tensorized Gaussian quadrature rule with $n$ points per dimension requires $N = n^q$ points.) 

In this paper, we describe one strategy to mitigate this cost via a procedure that combines statistical boosting ideas from theoretical computer science (see, e.g., \citep[Sec.~7.2]{mahoney2011RandomizedAlgorithms} and \citep[Sec.~2.3]{woodruff2014SketchingTool}) with bi-fidelity strategies in UQ. More precisely, our approach boosts on the randomness of a \textit{sketching} operator $\bs{S} \in \R^{m \times N}$ that is used to approximately solve \eqref{myleastsquares}:
\begin{align*}
  \doublehat{\bs{x}} &= \argmin_{\bs{x} \in \R^d} \| \bs{S} \bs{A} \bs{x} - \bs{S} \bs{b} \|_2^2.
\end{align*}
Without \textit{a priori} knowledge of $\bs{b}$, a deterministic sketch with $m < N$ generally is not robust to adversarial vectors $\bs{b}$ that result in a large residual for $\doublehat{\bs{x}}$ relative to the residual for $\bs{x}^\ast$. However, in general scenarios one can identify constructive \textit{probabilistic} models for $\bs{S}$ where sketches of near-optimal size, $m \gtrsim d\log d/(\epsilon \delta)$, ensure
\begin{align*}
  \|\bs{A} \doublehat{\bs{x}} - \bs{b} \|^2 \leq (1 + \epsilon) \|\bs{A} \bs{x}^\ast - \bs{b} \|^2 \hskip 5pt \textrm{with probability } \geq 1- \delta.
\end{align*}
We provide a more detailed discussion of existing sketching guarantees in section \ref{ssec:sketching}, in particular for \textit{row sketches}, for which computing $\bs{S} \bs{b}$ requires knowledge of only $m$ entries of $\bs{b}$, rather than all $N$ entries. While random sketching provides attractive guarantees when $m \ll N$, it is still random and hence is subject to randomness in performance, and ``failure'' events can occur with nonzero probability $\delta$. Naive statistical boosting mitigates this issue by generating several (say $L$) sketches and choosing the one that yields the smallest residual. However, this requires generating $L m$ entries of $\bs{b}$, which can be computationally expensive when each evaluation is an expensive PDE solve. Our approach attacks this problem in the sketch selection boosting phase by replacing $\bs{b}$ with an approximate, low-fidelity version from which collecting $L m$ samples is computationally feasible. Once a ``good'' sketch is identified in the boosting phase, we solve the sketched least squares problem using the corresponding sketch of the original data $\bs{b}$.

Thus, we assume availability of and leverage a \textit{low-fidelity} model $\widetilde{b}(\bs{p})$. For example, $\widetilde{b}$ may correspond to using a discretized PDE solver with a mesh coarser than the one which produces accurate realizations of $b$, or to model approximations such as Reynolds-averaged Navier Stokes solvers, or to solutions computed with arithmetic in lower precision compared to samples for $b$. Although $\widetilde{b}$ may be untrusted as a replacement for $b$, it can be used to extract some useful information about $b$, as is done in by-now standard multi-fidelity approaches \cite{peherstorfersurvey2018}. Throughout this paper, we assume the bi-fidelity setup, i.e., two levels of fidelity, and also that the cost of evaluating $\widetilde{b}$ is much less than the corresponding cost for $b$; both of these are common practical assumptions \cite{doostanstochastic2007,narayanstochastic2014,zhucomputational2014,fairbanks2020bi,newberry2022bi}.

\subsection{Contributions of this article}
The contributions of this article are as follows:
\begin{itemize}
  \item We propose a new bi-fidelity boosting  (BFB) algorithm to compute an approximation to $\bs{x}^\ast$. The procedure, given in Algorithm \ref{alg:BFQS}, computes the solution of a \textit{sketched} least squares problem, where the sketch matrix is identified by a boosting procedure on a low-fidelity data vector $\widetilde{\bs{b}}$. The sketching approach reduces the required sample complexity from $N$ evaluations of $b$ to $\sim d \log d$ samples of $b$, which can be a significant saving. The boosting procedure requires $\sim L d \log d$ evaluations of the low-fidelity model $\widetilde{b}$, where, in the language of statistical learning, $L$ is the number of weak learners used in the boosting procedure. When $\widetilde{b}$ costs substantially less than $b$, this cost for collecting the boosting data is negligible.
  
  \item We provide a theoretical analysis of BFB under certain assumptions, which provides quantitative bounds on the residual of the BFB solution $\hat{\bs{x}}_{\bfqs}$ relative to the full, computationally expensive solution $\bs{x}^\ast$ (see Theorems~\ref{thm:optimality} and \ref{thm:BFQS-error}).
  We also provide some asymptotic bounds on the correlation between the low- and high-fidelity solutions in a certain sense (see Theorem~\ref{thm:Gaussian}).
  Finally, we provide concrete computational strategies to ensure that the required assumptions of BFB hold (see Theorem~\ref{thm:sketches}). 
  
  \item We investigate the numerical performance of BFB when combined with several different sampling strategies and compare the performance to the corresponding sampling strategies without boosting.
  We also demonstrate using real-world problems that the assumptions required for BFB's theoretical analysis frequently hold in practice.
\end{itemize}

The idea of sketching for least squares solutions has a substantial history in the computer science and numerical linear algebra communities \cite{mahoney2011RandomizedAlgorithms, woodruff2014SketchingTool}. Our use of sparse row sketches of size $\sim d$ is identical to existing methods for leverage score-based \cite{mahoney2011RandomizedAlgorithms}, Gaussian-sketch based \cite{martinsson2020RandNLA}, and volume-maximizing sketching \cite{derezinski2018ReverseIterative,derezinski2018LeveragedVolume}. In addition, boosting for least squares problems is also not a new idea \cite{haberstichboosted2022}. However our combination of these approaches in a bi-fidelity setting is new to our knowledge, and our analysis in this bi-fidelity context provides novel, non-trivial insight into the algorithm performance.

The rest of this manuscript is organized as follows. 
Section~\ref{sec:preliminaries} introduces the notation we use and provides some background material on various sketching approaches in least squares approximation.
Section~\ref{sec:method} presents the BFB algorithm along with its theoretical analysis. 
Section \ref{sec:numerical-experiments} contains numerical experiments which illustrate various aspects of the BFB approach. 
We conclude the present study in Section \ref{sec:conclusion}.
The paper also contains several appendices.
Appendix~\ref{sec:lev-score-sampling-alg} provides a brief introduction to the sampling approach that we proposed in \citep{malik2022FastAlgorithms} and which we make use of in this paper.
Appendices~\ref{app:a} and \ref{app:a0} contain some proofs that have been left out of the main text.

\section{Preliminaries} \label{sec:preliminaries}

For the interest of clarity and completeness, we next introduce the notation used throughout the manuscript and introduce four sampling strategies to sketch the least squares problem (\ref{myleastsquares}), namely, sampling via column-pivoted QR, leverage scores, volume maximization, and Gaussian distribution.

\subsection{Notation} \label{sec:preliminaries-notation}

Matrices are denoted by bold upper-case letters (e.g., $\bs{A}$), vectors are denoted by bold lower-case letters (e.g., $\bs{x}$) and scalars by lower case regular and Greek letters (e.g., $a$ and $\alpha$).
Entries of matrices and vectors are indicated in parentheses. 
For example, $\bs{A}(i,j)$ is the entry on position $(i,j)$ in $\bs{A}$ and $\bs{a}(i)$ is the $i$th entry in $\bs{a}$.
A colon is used to denote all entries along a mode of a matrix. 
For example, $\bs{A}(i,:)$ is the $i$th row of $\bs{A}$ represented as a row vector.
For a set of indices $\Jc$, $\bs{A}(\Jc,:)$ denotes the submatrix $(\bs{A}(j, :))_{j \in \Jc}$ and $\bs{a}(\Jc)$ denotes the subvector $(\bs{a}(j))_{j \in \Jc}$. 

The \emph{compact SVD} of a matrix $\bs{A}$ takes the form $\bs{A} = \bs{U} \bs{\Sigma} \bs{V}^\top$, where $\bs{U}$ and $\bs{V}$ have $\rank(\bs{A})$ columns and $\bs{\Sigma}$ is of size $\rank(\bs{A}) \times \rank(\bs{A})$.
The \emph{pseudoinverse} of $\bs{A}$ is denoted by $\bs{A}^\dagger \defeq \bs{V} \bs{\Sigma}^{-1} \bs{U}^\top$.
For a matrix $\bs{U}$ with orthonormal columns, we use $\bs{U}_\perp$ to denote an orthonormal complement of $\bs{U}$, i.e., $\bs{U}_\perp$ is any matrix such that $[\bs{U},\;\; \bs{U}_\perp]$ is square and has orthonormal columns. We use $\bs{P}_{\bs{A}} \defeq \bs{A} \bs{A}^\dagger = \bs{U} \bs{U}^\top$ to denote the \emph{orthogonal projection} onto $\range(\bs{A})$,
where $\bs{U}=\orth(\bs A)$ is a(ny) matrix whose columns are an orthonormal basis for $\range(\bs A)$, e.g., via  the compact SVD or QR decomposition of $\bs A$. The determinant of $\bs{A}$ is denoted by $\det(\bs{A})$.
For a positive integer $n$, we use the notation $[n] \defeq \{1,2,\ldots,n\}$.
We use $\bs{a}_{\PP}$ to denote a vector $\bs{a} \neq \zerobf$ rescaled to unit length:
\begin{equation}
	\bm a_\PP = \frac{\bm a}{\|\bm a\|_2}.
\end{equation} 

We also introduce two notions of correlation: for given deterministic vectors $\bs{a}, \bs{b} \neq \bs{0}$, we define the correlation between them as the cosine of the angle separating them:
\begin{align*}
  \cor(\bs{a}, \bs{b}) \defeq \frac{\left\langle \bs{a}, \bs{b} \right\rangle}{\|\bs{a}\|_2\|\bs{b}\|_2},
\end{align*}
where $\langle\cdot,\cdot\rangle$ denotes the Euclidean inner product. We will also require Pearson's correlation coefficient, which is widely used in statistics.
For two (non-constant) random variables $X$ and $Y$ with bounded second moments defined on the same probability space, their correlation is defined as 
\begin{equation}
	\cor(X, Y) \defeq \frac{\E[(X-\E[X])(Y-\E[Y])]}{\sqrt{\V[X]\V[Y]}},
\end{equation} 
where $\E[\cdot]$ and $\V[\cdot]$ are, respectively, the mathematical expectation and variance operators. 
Note that our notation $\cor(\cdot,\cdot)$ is overloaded, operating differently on vectors and (random) scalars. The context of use in what follows should make it clear which definition above is used.

We will use the following notation to denote the minimum of the least squares objective in \eqref{myleastsquares}:
\begin{equation} \label{eq:r-definition}
	r(\bs{A}, \bs{b}) 
	\defeq \min_{\bs{x}} \| \bs{A} \bs{x} - \bs{b} \|_2 
	= \| \bs{A} \bs{x}^* - \bs{b} \|_2,
\end{equation}
where $\bs{x}^*$ is defined as in \eqref{myleastsquares}.

\subsection{Sketching of least squares problems}\label{ssec:sketching}

Solving the problem \eqref{myleastsquares} using standard methods (e.g., via the QR decomposition) costs $\mathcal{O}(Nd^2)$\footnote{In our context, we have $N > d$; see Assumption \ref{ass:b-range}.}.
When $N$ is large, this may be prohibitively expensive.
A popular approach to address this issue is to apply a \emph{sketch operator} $\bs{S} \in \R^{m \times N}$ where $m \ll N$ to both $\bs{A}$ and $\bs{b}$ in \eqref{myleastsquares} in order to reduce the size of the problem:
\begin{equation} \label{eq:xast-def}
  \hat{\bs{x}} \defeq \argmin_{\bs{x} \in \Rb^d} \left\| \bs{S} \bs{A} \bs{x} - \bs{S} \bs{b} \right\|_2.
\end{equation}
This approach has two benefits: (i) If $\bs{S}$ is a row-sketch, i.e., has only a small number of non-zero columns, then $\bs{S} \bs{b}$ requires knowledge of only a small number of entries of $\bs{b}$, and (ii) the cost of solving this smaller problem is $\mathcal{O}(md^2)$, a substantial reduction from $\mathcal{O}(Nd^2)$ when $m \ll N$.
Analogously to \eqref{eq:r-definition}, we will use the following to denote the least squares objective value for the approximate solution:
\begin{equation} \label{eq:r_S}
	r_{\bs{S}}(\bs{A}, \bs{b}) \defeq \| \bs{A} \hat{\bs{x}} - \bs{b} \|_2.
\end{equation}
The goal is for the approximation $\hat{\bs{x}}$ to yield a residual ``close'' to the optimal residual of the full problem \eqref{myleastsquares},
\begin{equation} \label{eq:rS-def}
  r(\bs{A}, \bs{b}) \approx r_{\bs{S}}(\bs{A}, \bs{b}) ,
\end{equation}
which is typically achieved if $m$ is ``large enough''. 
The following definition makes this more precise.
\begin{definition}[$(\e,\delta)$ pair condition] \label{def:epsilon-delta-condition}
	Let $\bm S\in\R^{m\times N}$ be a random matrix. 
	Given $\bm A\in\R^{N\times d}$, $\bm b \in\R^N$, and $\e, \delta>0$, the distribution of $\bm S$ is said to satisfy an $(\e, \delta)$ pair condition for $(\bm A, \bm b)$ if, with probability at least $1-\delta$, both conditions,
	\begin{equation} \label{mypair}
          \mathrm{rank}(\bm S \bm A) = \mathrm{rank}(\bm A) \hskip 10pt \textrm{and} \hskip 10pt r_{\bs{S}} (\bs{A}, \bs{b}) \leq (1+\varepsilon) \, r(\bs{A}, \bs{b}),
	\end{equation}
        hold simultaneously, where $r(\bs{A}, \bs{b})$ and $r_{\bs{S}} (\bs{A}, \bs{b})$ are defined as in \eqref{eq:r-definition} and \eqref{eq:r_S}, respectively.
\end{definition}

Note that one can only ask for the above condition with probability less than 1: For any sketch with $m < N$, there are vectors $\bs{b}$ for which the residual bound condition in \eqref{mypair} can be violated. Such a condition can be satisfied with $m < N$ samples; see sections \ref{sssec:lev_score_sampling}, \ref{sec:leveraged-volume-sampling}, and \ref{sec:gaussian-sketch}. Sketching operators $\bs{S}$ that sample a subset of the rows are of particular
interest in UQ since $\bs{S} \bs{b}$ in \eqref{eq:xast-def} then requires knowledge of only a subset of entries in the vector $\bs{b}$, meaning that fewer samples need to be collected.
In this paper, we consider three different sketching operators of this type, one of which is deterministic and two of which are random.
These are described in Sections~\ref{sssec:qr_sampling}--\ref{sec:leveraged-volume-sampling}.
Another popular sketching operator is the Gaussian sketching operator whose entries are appropriately scaled i.i.d.\ normal random variables.
Applying such a random matrix to $\bs{b}$ requires knowledge of all entries in $\bs{b}$.
While this makes the Gaussian sketch unsuitable for use in practice for quadrature sampling, we still consider it in some of our theoretical results since it is easier to analyze than the sampling-based sketches. Furthermore, since it is known to have excellent guarantees, it provides a nice baseline.
We introduce the Gaussian sketch in Section~\ref{sec:gaussian-sketch}.

Much research has been conducted over the last two decades on randomized algorithms in numerical linear algebra, including the problem of solving least squares problems.
We only cover the basics that are relevant for this paper.
For a more in-depth discussion, we refer the reader to the surveys in \cite{halko2011FindingStructure, mahoney2011RandomizedAlgorithms, woodruff2014SketchingTool, martinsson2020RandNLA} and the references therein.

\subsubsection{Sampling via column-pivoted QR decomposition} \label{sssec:qr_sampling}

Let $\bs{A}^\top \bs{P} = \bs{A}(\mathcal{J}, :)^\top = \bs{Q} \bs{R}$ be a column-pivoted QR (CPQR) decomposition where $\mathcal{J}$ is a length-$N$ permutation vector.
A simple deterministic heuristic for sampling $m$ rows from $\bs{A}$ is to simply choose those rows corresponding to the first $m$ entries in $\mathcal{J}$, i.e., $\bs{A}(\mathcal{J}(1:m), :)$.
This corresponds to applying a sketch $\bs{S} = (\bs{P}(:, 1:m))^\top$ to $\bs{A}$.
Such an approach has been used to sub-sample points from either tensor product quadratures~\cite{seshadri2017EffectivelySubsampled} or from random samples (approximate D-optimal design)~\cite{hadigol2018least,diaz2018sparse,guo_weighted_2018} in the context of least squares polynomial approximation. 

Recall that $\bs{A}$ is an $N \times d$ tall-and-skinny matrix.
When $m\le d$, the subsample is straightforward and just takes the first $m$ entries in $\mathcal J$ since the list $\mathcal J$ contains the entries in decreasing order of importance (as approximated by the column-pivoting algorithm). When $m>d$, the situation is more subtle since 
the remaining entries $\mathcal{J}(d+1:N)$ 
have no particular meaning and will 
not be useful in our row-sampling procedure.
To get around this, we use the heuristic in Algorithm~\ref{alg:QR-sampling-heuristic} in order to sample $m > d$ rows.
The heuristic chooses the first $d$ rows indices to be the entries in $\mathcal{J}(1:m)$ where $\mathcal{J}$ comes from the column-pivoted QR decomposition of $\bs{A}^\top$.
The rows with indices in $\mathcal{J}(1:m)$ are then removed from $\bs{A}$.
Another column-pivoted QR decomposition is then computed for the updated $\bs{A}^\top$, and the next set of $d$ rows is chosen to be the rows of $\bs{A}$ corresponding to the top-$d$ entries in the new permutation vector $\mathcal{J}$.
Once again, the chosen rows are removed from $\bs{A}$.
This procedure is repeated until $m$ rows have been chosen.
It is straightforward to formulate a sampling matrix $\bs{S}$ such that $\bs{S}\bs{A} = \bs{A}_s$, where $\bs{A}_s$ is the output of Algorithm~\ref{alg:QR-sampling-heuristic}. 

\begin{algorithm}[ht]
	\caption{Heuristic for sampling via column-pivoted QR decomposition\label{alg:QR-sampling-heuristic}}
	\KwIn{$\bs{A}$: design matrix; $m$: desired number of row samples}
	\KwOut{$\bs{A}_s$: matrix containing $m$ rows of $\bs{A}$}
	\begin{algorithmic}[1]
		
		\STATE Initialize $\bs{A}_s$ to an empty matrix: $\bs{A}_s = [\,]$
		\WHILE{$m > 0$}{
			\STATE Compute column-pivoted QR of $\bs{A}^\top$: $\bs{A}(\mathcal{J}, :)^\top = \bs{Q} \bs{R}$
			\STATE Let $k = \min(d,m)$
			\STATE Append top-$k$ rows from $\bs{A}$ to $\bs{A}_s$: $\bs{A}_s = [\bs{A}_s; \, \bs{A}(\mathcal{J}(1:k), :)]$
			\STATE Remove top-$k$ rows from $\bs{A}$: $\bs{A} = \bs{A}(\mathcal{J}(k+1:\text{end}), :)$
			\STATE $m = m - k$
		}
		\ENDWHILE
		\RETURN $\bs{A}_s$
	\end{algorithmic}
\end{algorithm}

Since the approach in Algorithm~\ref{alg:QR-sampling-heuristic} is deterministic, it cannot satisfy guarantees of the form in Definition~\ref{def:epsilon-delta-condition}.
However, for the case $m=d$ it is possible to prove bounds on the condition number of $\bs{A}(\mathcal{J}(1:d),:)$; see Lemma~2.1 in \citep{seshadri2017EffectivelySubsampled} for details.

\subsubsection{Leverage score sampling} \label{sssec:lev_score_sampling}

Let $\bs{A} = \bs{U} \bs{\Sigma} \bs{V}^\top$ be a compact SVD.
The \emph{leverage scores} of $\bs{A}$ are defined as 
\begin{equation} \label{eq:leverage-scores-definition}
	\ell_i(\bs{A}) \defeq \|\bs{U}(i,:)\|_2^2 \;\;\;\; \text{for } i \in [N].
\end{equation}
They take values in the range $\ell_i(\bs{A}) \in [d/N, 1]$ and indicate how important each row of $\bs{A}$ is in a certain sense.
The matrix $\bs{U}$ can be replaced with any matrix whose columns form an orthonormal basis for $\range(\bs{A})$ without impacting the definition in \eqref{eq:leverage-scores-definition} \citep[Sec.\ 2.4]{woodruff2014SketchingTool}.
The \emph{coherence} of $\bs{A}$ is defined as
\begin{equation} \label{eq:coherence}
	\gamma(\bs{A}) \defeq \max_{i \in [N]} \ell_i(\bs{A}).
\end{equation}
It takes values in the range $\gamma(\bs{A}) \in [d/N, 1]$; it is maximal when one of the leverage scores is 1 and minimal when all leverage scores are equal to $d/N$.
Let $r\defeq \sum_i \ell_i(\bs{A})$. The leverage score sampling distribution of $\bs{A}$ is defined as
\begin{equation}
	p_i(\bs{A}) \defeq \frac{\ell_i(\bs{A})}{r} \;\;\;\; \text{for } i \in [N],
\end{equation}
which is indeed a probability distribution as $\ell_i(\bs{A})>0$. Let $f : [m] \rightarrow [N]$ be a random map such that each $f(j)$ is independent and $\Pb\{ f(j) = i \} = p_i(\bs{A})$ for each $j \in [m]$.
The leverage score sampling sketch $\bs{S} \in \Rb^{m \times N}$ is defined elementwise via
\begin{equation} \label{eq:lev-score-sketch}
	\bs{S}_{ji} = \frac{\Ind\{ f(j) = i \}}{\sqrt{m p_{f(j)} (\bs{A})}} \;\;\;\; \text{for } (j,i) \in [m] \times [N],
\end{equation}
where $\Ind \{A\}$ is the indicator function which is 1 if the random event $A$ occurs and zero otherwise.
Algorithms and theory for leverage score sampling have been developed in a number of papers; see e.g., \citep{drineas2006SamplingAlgorithms, drineas2008RelativeerrorCUR, drineas2011FasterLeast, mahoney2011RandomizedAlgorithms, larsen2020PracticalLeverageBased} and references therein.
The distribution for the leverage score sketch in \eqref{eq:lev-score-sketch} satisfies an $(\varepsilon, \delta)$ condition for $(\bs{A}, \bs{b})$ if 
\begin{equation} \label{eq:leverage-score-sampling-complexity}
	m \gtrsim d \log(d / \delta) + d / (\varepsilon \delta);
\end{equation}
see Theorem~\ref{thm:sketches} for a more detailed and slightly stronger statement.

Choosing $p_i(\bs{A}) = 1/N$ results in \emph{uniform sampling}.
For general matrices, there are no useful guarantees when sampling uniformly in this fashion.
However, if $\bs{A}$ has low coherence, then uniform sampling will be close to the leverage score sampling distribution and guarantees similar to those for leverage score sampling hold.
More precisely, if $\ell_i(\bs{A}) \leq C d/N$ for some constant $C \geq 1$, then uniform sampling satisfies an $(\varepsilon, \delta)$ condition for $(\bs{A}, \bs{b})$ if $m$ is chosen as in \eqref{eq:leverage-score-sampling-complexity} (this is a direct consequence of, e.g., Theorem~6 in \citep{larsen2020PracticalLeverageBased}). 
Notice that the difference from sampling according to the exact leverage scores is that there now is an additional constant $C$ hidden in the lower bound on $m$.

In addition to a parsimonious sampling of $\bs{b}$, the computational complexity of the sketched least squares approach in \eqref{eq:xast-def} is a consideration. Direct sampling of the leverage score distribution via the formula \eqref{eq:leverage-scores-definition} requires a matrix decomposition (e.g., QR or SVD), which costs $\mathcal{O}(Nd^2)$ effort, the same effort required to solve the original least squares problem.
\citet{drineas2012FastApproximation} propose a procedure for computing leverage score estimates with cost $\mathcal{O}(N d \log N)$ for any matrix $\bs{A}$.
When $\bs{A}$ has particular structure it is possible to improve this considerably.
\citet{malik2022FastAlgorithms} propose such a method for the case when the multivariate basis functions $\psi_j$ in \eqref{eq:function-approximation} are certain products of one-dimensional functions, which corresponds to impose certain structure on the subspace $V$.
In the polynomial approximation setting, those structural conditions are satisfied by a large family of subspaces, including the popular tensor product, total degree, and hyperbolic cross spaces.
For example, if the multivariate basis polynomials for $q$-dimensional inputs correspond to polynomials of at most degree $k$ in each dimension and use $n$ grid points per dimension (in which case $\bs{A}$ has $N=n^q$ rows), then the total cost of our method is at most $\mathcal{O}(q n k^2 + m q)$ for drawing $m$ samples.
This sampling approach is an ingredient in our method, so we describe the key aspects of how this sampling approach works in Appendix~\ref{sec:lev-score-sampling-alg} and refer the reader to \citep{malik2022FastAlgorithms} for a more comprehensive treatment.

\subsubsection{Leveraged volume sampling} \label{sec:leveraged-volume-sampling} \label{sssec:lev_vol_sampling}

Volume sampling is a technique that samples a set $\Jc \subset [N]$ of $m$ row indices of $\bs{A}$ with probability proportional to the squared volume of the parallelepiped spanned by the columns of the submatrix $\bs{A}(\Jc,:)$, i.e., $\Pb(\Jc) \propto \det(\bs{A}(\Jc,:)^\top \bs{A}(\Jc,:))$.
This means that, unlike for leverage score sampling, the rows are not sampled independently. 
This has several benefits, including that the sketched least square solution $\bs{A}(\Jc,:)^\dagger \bs{b}(\Jc)$ is correct in expectation \citep[Prop.~7]{derezinski2017UnbiasedEstimates}: $\Eb[\bs{A}(\Jc,:)^\dagger \bs{b}(\Jc)] = \bs{A}^\dagger \bs{b}$. 
Leverage score sampling, by contrast, may produce a biased estimate of the solution vector.
Despite the apparent issue of sampling from a combinatorial number of subsets of $[N]$, there are algorithms for volume sampling that run in polynomial time.
\citet{derezinski2018ReverseIterative} propose two such algorithms, RegVol and FastRegVol. 
RegVol runs in $\mathcal O((N-m+d)Nd)$ time, and FastRegVol runs in $\mathcal O((N+\log(N/d) \log(1/\delta))d^2)$ time with probability at least $1-\delta$.
The dependence on $N$ can be prohibitive in quadrature sampling since the number of (tensor-product) quadrature points $N$ is exponential in the number of variables.

\citet{derezinski2018LeveragedVolume} propose \emph{leveraged} volume sampling which improves on standard volume sampling in several ways. 
Importantly, it still retains the correctness in expectation but allows for more efficient sampling.
In particular, the cost of sampling does not depend on $N$.
Unlike standard volume sampling, the sketch distribution satisfies an $(\varepsilon, \delta)$ condition for $(\Abf, \ybf)$ if $m \gtrsim d \log(d / \delta) + d / (\varepsilon \delta)$, which is on par with what leverage score sampling requires for such guarantees.
Leveraged volume sampling has two stages.
In the first stage, $\mathcal O(d^2)$ rows are chosen from $\bs{A}$ using a combination of leverage score sampling and rejection sampling.
After that, the $\mathcal O(d^2)$ subset is further reduced to $\mathcal{O}(d \log(d / \delta) + d / (\varepsilon \delta))$ via standard volume sampling.
In the experiments, we use FastRegVol from \citep{derezinski2018ReverseIterative} for the second step.
When FastRegVol is used, the cost of leveraged volume sampling is $\mathcal O(((d^2+m) d^2 + m C_\text{samp}) \log(1/\delta))$, where $C_\text{samp}$ is the cost of drawing one row index of $\bs{A}$ using leverage score sampling.
As discussed in Section~\ref{sssec:lev_score_sampling}, the the cost $C_\text{samp}$ of leverage score sampling can be reduced drastically in our setting by using the structured sampling techniques from \cite{malik2022FastAlgorithms}.

\subsubsection{Gaussian sketching operator} \label{sec:gaussian-sketch}

The Gaussian sketching operator $\bs{S} \in \Rb^{m \times N}$ has entries that are i.i.d.\ Gaussian random variables with mean zero and variance $1/m$.
The Gaussian sketch satisfies an $(\varepsilon, \delta)$ condition if $m \gtrsim (d/\varepsilon) \log(d / \delta)$.
These results also extend to the case when the entries of $\bs{S}$ are sub-Gaussian; see Theorem~\ref{thm:sketches} for further details.

The main benefit of the Gaussian sketching operator is that it allows for simple and precise theoretical analysis of procedures that use sketching as a subroutine \citep[Remark~8.2]{martinsson2020RandNLA}.
This is our motivation for considering the Gaussian sketch in this paper. 
Computationally, it is not efficient to use Gaussian sketching for least squares problems.
The reason is that computing $\bs{S} \bs{A}$ costs $\mathcal{O}(mNd)$ which is more than the $\mathcal{O}(N d^2)$ cost of solving the original least squares problem (recall that $m > d$).
As discussed earlier, an additional issue in bi-fidelity estimation is that computing $\bs{S} \bs{b}$ requires knowledge of all elements of $\bs{b}$ which is prohibitively expensive when that vector contains high-fidelity data.

\subsection{Bi-fidelity problems}\label{ssec:bi-fidelity}

The main goal of this paper is to propose a strategy that improves the accuracy of sketching via a boosting procedure that employs a full vector $\tilde{\bs{b}}$ corresponding to an inexpensive low-fidelity approximation to $\bs{b}$.

Bi-fidelity frameworks assume the availability of a low-fidelity simulation $\widetilde{\mathcal{T}}$; that is, a map $\widetilde{\mathcal{T}}: \Rb^q \rightarrow \Rb$ such that $\widetilde{\mathcal{T}}$ is parameterically correlated with $\mathcal{T}$ in some sense, but need not be close to $\mathcal{T}$ in terms of sampled values. Such properties arise, for example, in parametric PDE contexts when $\widetilde{\mathcal{T}}$ arises as the discretized PDE solution operator on a spatial mesh that is coarser (and hence less trusted) than the mesh corresponding to $\mathcal{T}$. The decreased accuracy/trustworthiness of $\widetilde{\mathcal{T}}$ is balanced by its decreased cost, so that employment of $\widetilde{\mathcal{T}}$ may not furnish precise high-fidelity information, but may provide useful knowledge in terms of dependence on the parameter $\bs{p}$ with substantially reduced cost.

In the context of constructing our emulator \eqref{myleastsquares}, our core assumption is that the low-fidelity operator $\widetilde{\mathcal{T}}$ is cheap enough so that full exploration of the response over the sampled parameter set $\{\bs{p}_i\}_{i\in [N]}$ is more computationally feasible, resulting in a vector $\tilde{\bs{b}} \in \Rb^{N}$ with low-fidelity entries 
\begin{equation} \label{eq:vector-bt-construction}
	\tilde{\bs{b}}(n) = \sqrt{w_n} \widetilde{\mathcal{T}}(\bs{p}_n).
\end{equation}
Of course, one may propose constructing the emulator $\mathcal{T}$ in \eqref{myleastsquares} by simply replacing $\bs{b}$ by $\tilde{\bs{b}}$, but this restricts the accuracy of the emulator $\mathcal{T}$ to the potentially bad accuracy of $\widetilde{\mathcal{T}}$. 
In this paper, we propose a more sophisticated use of $\tilde{\bs{b}}$, in conjunction with a single sparse sketch of $\bs{b}$, that retains some accuracy characteristics of $\bs{x}^*$.

\section{Bi-fidelity boosting (BFB) in sketched least squares problems}\label{sec:method}

In practice, one often requires the probability of successfully obtaining a good approximation $\bs{x}^\ast$ associated with a random sketch from section \ref{ssec:sketching} to be sufficiently close to $1$, and one way to achieve this with fixed sketch size is through a boosting procedure. 
Assuming the availability of a collection of sketching matrices $\{\bs{S}_{\ell}\in\R^{m\times N}\}_{\ell \in [L]}$, one computes the residual for the $\bs{S}_\ell$-sketched solution (i.e., $\|\bs{A} (\bs{S}_\ell \bs{A})^\dagger (\bs{S}_\ell \bs{b})  - \bs{b}\|_2$) for each $\bm S_\ell$ and then selects the one that yields the smallest residual for use. 
Even if each sketch is sparse, this straightforward procedure inflates the required sampling cost of the forward model $\mathcal{T}$ by the factor $L$, which may be computationally prohibitive. 
To ameliorate this boosting cost, we employ a bi-fidelity strategy.

In Section~\ref{sec:proposed-algorithm} we present our proposed algorithm for quadrature sampling which leverages sketching BFB.
Sections \ref{ssec:theory-pre} and \ref{ssec:theory-asymp} give our pre-asymptotic and asymptotic analysis results, respectively. We collect some preliminary technical results in section \ref{sec:theoretical-analysis-prelims}, and prove our pre-asymptotic results in section \ref{sec:theoretical-analysis-optimality-coeff}. The asymptotic result is proven in Appendix \ref{app:a}. We end with section \ref{sec:sketch-guarantees} that provides results for random sketches achieving the $(\epsilon, \delta)$ condition in Definition \ref{def:epsilon-delta-condition}.

\subsection{Proposed algorithm} \label{sec:proposed-algorithm}
 
A distinguishing feature of the least squares problem in our setup is that full information of the high-fidelity data $\bm b$ is unaffordable due to computational restrictions; instead, we can only afford to generate a small number of entries of $\bs{b}$.
Meanwhile, the low-fidelity data vector $\tilde{\bm b}\in\R^N$ that exhibits some type of correlation with $\bm b$ is readily available for repeated use. (This correlation-like condition is quantifying through the parameter $\nu$ introduced in Theorem \ref{thm:optimality}.)
We propose a modified boosting procedure, where the boosting phase of a sketched least squares problem replaces high-fidelity data with low-fidelity data to find the ``best'' sketching operator and then employs this best sketch directly with high-fidelity data to compute an approximate least squares solution. 
This procedure is outlined in Algorithm~\ref{alg:BFB}.

\begin{algorithm}[ht]
		\DontPrintSemicolon
		\caption{Bi-Fidelity Quadrature Boosting (BFB)\label{alg:BFB}}
		\KwIn{design matrix $\bm A$,  low-fidelity vector $\tilde{\bm b}$, method for computing entries of the high-fidelity vector $\bm b$, collection of sketches for boosting $\{\bs{S}_{\ell}\}_{\ell\in [L]}$}
		\KwOut{an approximate solution $\hat{\bm x}_\bfqs$ to \eqref{myleastsquares}}
		\begin{algorithmic}[1]
			\FOR{$\ell\in [L]$}{
			\STATE compute the $\ell$-th sketched solution $\hat{\bm x}_\ell$ using the low-fidelity data: 
			\begin{equation}
				\hat{\bm x}_\ell = \argmin_{\bm x\in\R^d}\left\| \bm S_\ell\bs{A} \bs{x} - \bm S_\ell\tilde{\bs{b}} \right\|_2 
			\end{equation}
			}
			\ENDFOR
			\STATE find the best low-fidelity sketch index $\ell^*$ using boosting:
			\begin{equation}
				\ell^* = \argmin_{\ell\in [L]}\|\bm A\hat{\bm x}_\ell - \tilde{\bs{b}}\|_2			
			\end{equation} \label{line:check-low-fid-sketch}
			\STATE use sketch $\bm S_{\ell^*}$ to compute an approximate solution to \eqref{myleastsquares}: 
			\begin{equation}
				\hat{\bm x}_\bfqs = \argmin_{\bm x\in\R^d}\left\| \bm S_{\ell^*}\bs{A} \bs{x} - \bm S_{\ell^*}\bs{b} \right\|_2
				\;\;\;\; \text{\tcp*{Requires computing $m$ entries of $\bs{b}$}}
			\end{equation}
		\end{algorithmic}
		 
		\label{alg:BFQS}
\end{algorithm}

The oracle sketch in this scenario is the one identified by the boosting strategy operating directly on the high-fidelity least squares problem, which is computationally unaffordable:
\begin{equation} \label{eq:l_starstar}
  \ell^{**} = \argmin_{\ell\in [L]}\|\bm{A} \doublehat{\bm x}_\ell-\bm{b}\|_2^2, \;\;\;\; \text{where } \doublehat{\bm x}_\ell = \argmin_{\bm x\in\R^d}\|\bm S_\ell\bm A\bm x-\bm S_\ell\bm b\|_2.
\end{equation}
In the coming sections 
we will theoretically investigate the sketch transferability between high- and low-fidelity boosting, i.e., when the residual associated to $\hat{\bm x}_\bfqs$, the solution produced by Algorithm \ref{alg:BFQS}, is comparable to the residual associated to $\hat{\bm x}_{\ell^{**}}$.

We divide our analysis into two cases: Our first analysis frames performance of Algorithm \ref{alg:BFQS} in terms of an \textit{optimality coefficient}, defined in \eqref{myeq1}, which measures the quality of the least squares residual for a particular sketch $\bs{S}$; we provide pre-asymptotic analysis with quantitative results that provides qualitative guidance on how the BFB algorithm behaves in terms of the tradeoff in the number of sketches $L$ versus the optimality coefficient (see the discussion following Theorem \ref{thm:BFQS-error}). Our second theoretical result is an asymptotic analysis with Gaussian sketches that confirms the intuition that the probabilistic correlations between the low- and high-fidelity random sketches is high when ${\bm b}$ and $\tilde{\bm b}$ have high vector correlations (see the discussion around Theorem \ref{thm:Gaussian}).

For analysis purposes we make the following assumption.
\begin{assumption} \label{ass:b-range}
	Assume that neither $\tilde{\bm b}$ nor $\bm{b}$ lie in $\range(\bs{A})$, i.e., we assume $\tilde{\bm{b}}, \bm{b} \not\in \range(\bs{A})$.
\end{assumption}
This is a reasonable assumption. 
If $\bs{b} \in \range(\bs{A})$, then it would be possible to solve the high-fidelity least squares problem exactly by sampling $m = d$ linearly independent rows of $\bs{A}$ and the corresponding rows of $\bs{b}$.
In this case, it is therefore easy to solve \eqref{myleastsquares} and only requires accessing $d$ rows of $\bs{b}$.
Similarly, if $\tilde{\bs{b}} \in \range(\bs{A})$ then it would be easy to compute a sketch $\bs{S}_\ell$ which only samples $m = d$ rows and achieves zero error in Line~\ref{line:check-low-fid-sketch} of Algorithm~\ref{alg:BFB}, therefore making the boosting procedure vacuous.

\subsection{Pre-asymptotic analysis via optimality coefficients}\label{ssec:theory-pre}
We introduce the following measure of relative error difference between the sketched and optimal solutions:
\begin{equation} \label{myeq1}
	\mu_{\bm A}(\bm b, \bm S) 
	\defeq \sqrt{\frac{r^2_{\bm S}(\bm A, \bm b)-r^2(\bm A, \bm b)}{r^2(\bm A, \bm b)}} 
        \stackrel{(\ast)}{=} \frac{\|(\bm S\bm Q)^\dagger\bm S\bm Q_\perp\bm Q^T_\perp\bm b\|_2}{\|\bm Q_\perp\bm Q^T_\perp\bm b\|_2},
\end{equation}
where $\bm Q=\orth(\bm A)$, and the second equality marked $(\ast)$ is valid if $\rank(\bs{S} \bs{A}) = \rank(\bs{A})$, which we establish in Lemma~\ref{lemma:diff-of-residuals}.
For notational simplicity we usually drop the subscript and write $\mu(\bm b, \bm S)$ when $\bm A$ is clear from context, but we emphasize that $\mu$ does depend on $\bm A$.
Note that $r(\bm A, \bm b) = \|\bm Q_\perp\bm Q^T_\perp\bm b\|_2 > 0$ due to Assumption~\ref{ass:b-range}, so the denominator in \eqref{myeq1} is nonzero. 
We call $\mu$ the \emph{optimality coefficient}. 
Smaller values of $\mu$ are better in practice: $\mu = 0$ implies the sketch achieves perfect reconstruction of the data relative to the full least squares solution.

We provide two main theoretical results which shed light on the performance of Algorithm~\ref{alg:BFB} from two different perspectives.
The first result shows that with an appropriate choice of the sketches $\{\bs{S}_\ell\}_{\ell \in [L]}$, Algorithm~\ref{alg:BFB} produces a solution whose relative error is close to that of the oracle sketch solution in \eqref{eq:l_starstar}.
Note that it would be straightforward to provide such guarantees if 
$r_{\bm S}(\bm A, \tilde{\bm b}) \le r_{\bm S'}(\bm A, \tilde{\bm b})$ implied $r_{\bm S}(\bm A, \bm b) \le r_{\bm S'}(\bm A, \bm b)$,
in which case $\ell^* = \ell^{**}$.
This may happen, for instance, when $\tilde{\bm b}$ and $\bm b$ differ by a scaling. 
This monotone property of $r$ when replacing $\bm b$ with $\tilde{\bm b}$ is unfortunately unlikely to hold in practice.  
Our result, which appears in Theorem~\ref{thm:optimality}, identifies alternative conditions that ensure $\bm S_{\ell^*}$ is a ``good'' sketch for the high-fidelity data. 

\begin{theorem}\label{thm:optimality}
Fix a positive integer $L$ and suppose $\delta, \varepsilon \in (0,1]$.
If $\{\bm S_\ell\}_{\ell\in [L]}$ is a sequence of i.i.d.\ random matrices whose distribution is an $(\e, \frac{\delta}{L})$ pair for $(\bm Q, \bm h)$, where 
\begin{equation} \label{myh}
	\bm h \defeq \left((\bm P_{\bm Q_\perp}\bm b)_\PP-(\bm P_{\bm Q_\perp}\tilde{\bm b})_\PP\right)_\PP \quad\text{and}\quad \bm Q \defeq \orth(\bm A),
\end{equation}
then with probability at least $1-\delta$, 
\begin{equation} \label{bound1}
	\mu(\bm b, \bm S_{\ell^*})\leq \mu(\bm b, \bm S_{\ell^{**}}) + 2\sqrt{6(1-\nu)\e},
\end{equation}
where $\nu$ denotes the absolute correlation coefficient between $\bm P_{\bm Q_\perp}\bm b$ and $\bm P_{\bm Q_\perp}\tilde{\bm b}$:
\begin{equation} \label{mynu}
  \nu \defeq \left|\cor(\bm P_{\bm Q_\perp}\bm b, \bm P_{\bm Q_\perp}\tilde{\bm b})\right|.
\end{equation}
  In addition, on the event where \eqref{bound1} is true, we also have that \eqref{mypair} holds with ${\bm S} = {\bm S}_\ell$ for every $\ell \in [L]$.
\end{theorem}

Theorem~\ref{thm:optimality} shows that if a sketch satisfies an $(\e,\delta/L)$ condition for the pair $\bs{Q}$ and an element $\bs{h}$ of $\range(\bs{Q}_\perp)$, then we are able to prove bounds on the low-fidelity boosted optimality coefficient $\mu(\bs{b}, \bs{S}_{\ell^\ast})$ relative to the oracle high-fidelity boosted optimality coefficient $\mu(\bs{b}, \bs{S}_{\ell^{\ast\ast}})$.
This is quite a general statement that accommodates a wide range of sketching operators.
The condition on the operators $\{\bs{S}_\ell\}_{\ell \in [L]}$ is, for example, satisfied by all sketching operators in Sections~\ref{sssec:lev_score_sampling}--\ref{sec:gaussian-sketch} when the embedding dimension $m$ is sufficiently large.
More precise statements for the leverage score and Gaussian sketches are provided in Theorem~\ref{thm:sketches}.

In order to achieve a good approximate solution when applying sketching techniques in least squares problems the sketching operator must preserve the relevant geometry of the problem.
In particular, it is key that $\bs{Q}$ and $\bs{P}_{\bs{Q}_{\perp}} \bs{b}$ remain roughly orthogonal after the sketching operator has been applied.
This importance of preserving $\bs{P}_{\bs{Q}_{\perp}} \bs{b}$ in the sketching phase when $\bm b$ is replaced by low-fidelity data $\tilde{\bm b}$ manifests in Theorem \ref{thm:optimality} through the correlation parameter $\nu$.

\begin{remark}
  Equation \eqref{bound1} suggests that $\bm S_{\ell^*}$ is ``good'' when $\nu$ is large. This explicitly requires high parametric correlation between the portions of $\bm{b}$ and $\tilde{\bm b}$ that lie orthogonal to the range of $\bm A$. A more subtle sufficient condition ensuring large $\nu$ is furnished by our discussion following Proposition \ref{prop:cor}, which provides a lower bound for $\nu$ in terms of other parameters. 
\end{remark}

Theorem \ref{thm:optimality} does not provide a concrete strategy for how the sketches used in boosting are chosen or constructed. However, near-optimal sketches (in particular satisfying our required $(\epsilon,\delta)$ pair condition) are known to be produced through the well-known randomized approaches discussed in sections \ref{sssec:lev_score_sampling}-\ref{sec:gaussian-sketch}. Precise statements for such sketch estimates are given later in by Theorem \ref{thm:sketches} in section \ref{sec:sketch-guarantees}, but it is appropriate for us to establish here that combining Theorem \ref{thm:optimality} with good sketching techniques results in explicit and illuminating theory for Algorithm \ref{alg:BFQS}. In particular, one expects a tradeoff between the values of $\nu$ and $L$: boosting with a large number $L$ of sketches should work up to a threshold determined by the amount of correlation between ${\bm b}$ and ${\tilde{\bm b}}$. I.e., any accuracy gained by BFB should be limited by how correlated the low- and high-fidelity models are, and one expects this to manifest in a relationship between $L$ and $\nu$. The theory we develop below reveals this tradeoff. We focus on generating the sketches $\{\bm S_\ell\}_{\ell \in [L]}$ through leverage score sampling, as described explicitly by \eqref{eq:lev-score-sketch} in section \ref{sssec:lev_score_sampling}. We briefly discuss afterward that one could generalize the result to more general sketches.
\begin{theorem}\label{thm:BFQS-error}
  Let $\delta, \epsilon \in (0, 1/2)$ and $L \in \N$ be chosen, and assume
  \begin{align}\label{eq:d-bound}
    d \leq \frac{\delta}{4} \exp\left(\frac{2}{35 \epsilon \delta}\right).
  \end{align}
  Now consider Algorithm \ref{alg:BFQS}, where $\{\bm S_\ell\}_{\ell \in [L]}$ are iid samples of a leverage score sketching operator defined in \eqref{eq:lev-score-sketch}, with the sampling requirement 
      \begin{align}\label{eq:m-bound}
        m\geq \frac{4 d L}{\epsilon \delta}.
      \end{align}
      Then each $\bm S_{\ell}$ satisfies an $(\epsilon/L, \delta/2)$ condition for the pair $({\bm Q}, {\bm h})$, and with probability at least $1-\delta$, we have
      \begin{align}\label{eq:residual-bound-1}
        r^2_{\bs{S}_{\ell^{\ast}}}(\bs{A}, \bs{b}) \leq \left[1 + \frac{\epsilon}{L} \tau \right] r^2(\bs{A},\bs{b}),
      \end{align}
      where
      \begin{align*}
        \tau = \tau(\epsilon, \delta, \nu, L) = 24 L (1-\nu) + \frac{\delta}{2}\left(1 + 4\sqrt{6(1-\nu) \epsilon}\right).
      \end{align*}
\end{theorem}
The results above give explicit behavior of the BFB residual via a concrete sketching strategy for Algorithm \ref{alg:BFQS}. Note in particular that the sampling requirement $m = \mathcal{O}(L/\epsilon)$ in \eqref{eq:m-bound} means that \textit{without} boosting and simply generating one sketch ${\bm S}$ according to \eqref{eq:m-bound}, which requires $m$ high-fidelity samples (equivalent to the number from BFB), we expect that the residual from this one sketch behaves like
\begin{align*}
  r^2_{\bm S}({\bm A}, {\bm b}) \sim \left(1 + \frac{\epsilon}{L} \right) r^2({\bm A}, {\bm b}).
\end{align*}
Comparing the above to \eqref{eq:residual-bound-1}, note that the only difference is the appearance of $\tau$, and hence we expect BFB to be useful (compared to an equivalent number of high-fidelity samples devoted to a non-boosting strategy) when $\tau \leq 1$, which requires,
\begin{align*}
  L \lesssim \frac{1}{1-\nu}.
\end{align*}
I.e., boosting with $L$ sketches is useful in BFB up to a threshold $\sim 1/(1-\nu)$. Boosting with \textit{more} than this threshold level of sketches causes the error bound to saturate at a level determined by $1-\nu$. Since 
$\nu$ is the correlation between the $\mathrm{range}(\bs{A})$-orthogonal components of ${\bm b}$ and $\tilde{\bm b}$, we conclude that highly correlated range-orthogonal residuals (large values of $\nu$ very close to 1) are optimal for BFB in the sense that sketching with large $L$ will be effective.

A second observation we make is that the the $m \sim L$ requirement \eqref{eq:m-bound} is theoretically suboptimal. In particular, we show in Theorem \ref{thm:sketches} that stronger coherence-like conditions on the matrix $\bs{A}$ imply that leverage score sketching with $m \sim \log L$ is sufficient to achieve the requisite $(\epsilon/L, \delta)$ condition, see \eqref{lop_more} in Theorem \ref{thm:sketches}. We also note that Gaussian sketches only require $m \sim \log L$ samples (see \eqref{subgau}), and one can achieve the $(\epsilon,\delta)$ condition \textit{on average} using $m \sim \log L$ samples (see, e.g., \cite[Equation (2.18)]{malik2022FastAlgorithms}. Finally, if \eqref{eq:d-bound} is violated, then indeed $m \sim \log L$ (see \eqref{lop} and the intermediate computation in \eqref{eq:d-bound-use}) for leverage score sketches. Thus, we expect in practice that $m \sim \log L$ samples are sufficient.

We give the proof of theorem \ref{thm:BFQS-error} below to demonstrate how it relies on Theorem \ref{thm:optimality}; we will prove Theorem \ref{thm:optimality} in the coming sections.
\begin{proof}[Proof of Theorem \ref{thm:BFQS-error}]
  We start by making two conclusions from the conditions \eqref{eq:d-bound} and \eqref{eq:m-bound}. First, under these conditions,
  \begin{align}\label{eq:d-bound-use}
    35 \log \left(\frac{4 d}{\delta}\right) \leq \frac{2}{\epsilon \delta} \hskip 10pt \Longrightarrow \hskip 10pt m \geq d \max \left\{ 35 \log\left(\frac{4 d L}{(\delta/2)}\right), \frac{2 L}{\epsilon (\delta/2)} \right\},
  \end{align}
  implying that condition \eqref{lop} holds, so that result \ref{itm:lev-score} from Theorem \ref{thm:sketches} guarantees that the distribution from which the $\bs{S}_\ell$ sketches are drawn satisfies and $(\epsilon, \frac{\delta}{2 L})$ condition. Thus, theorem \ref{thm:optimality} states that there is an event $E_1$ such that
  \begin{align}\label{eq:E1}
    \mathrm{Pr}(E_1) \geq 1 - \delta/2, \hskip 10pt \textrm{On event $E_1$, then \eqref{bound1} holds.}
  \end{align}
  The above is our first conclusion. For our second conclusion, we note that \eqref{eq:m-bound} and \eqref{eq:d-bound} imply,
  \begin{align*}
    m \geq \frac{2 d}{\left[\frac{\epsilon}{L} \left(\frac{\delta}{2}\right)^{1-1/L} \right] \left(\frac{\delta}{2}\right)^{1/L}},
  \end{align*}
  so that again we satisfy \eqref{lop} (employing a variation of the argument \eqref{eq:d-bound-use}), and so by Theorem \ref{thm:sketches}, the distribution from which $\bs{S}_{\ell}$ is drawn satisfies an $(\tilde{\epsilon}, \tilde{\delta})$ condition for $(\bs{A}, \bs{b})$, where,
  \begin{align*}
    \tilde{\epsilon} &\defeq \frac{\epsilon}{L} \left(\frac{\delta}{2}\right)^{1-1/L},  & 
    \tilde{\delta} &\defeq \left(\frac{\delta}{2}\right)^{1/L}.
  \end{align*}
  Therefore with probability at least $1 - \tilde{\delta}$, 
  \begin{align*}
    r^2_{{\bm S}_\ell}(\bs{A}, \bs{b}) \leq (1 + \tilde{\epsilon}) r^2(\bs{A}, \bs{b}),
  \end{align*}
  so that a union bound implies that there is an event $E_2$ on which our second conclusion holds:
  \begin{align}\label{eq:E2}
    \mathrm{Pr}(E_2) \geq 1 - \left(\tilde{\delta}\right)^L = 1 - \delta/2 \hskip 10pt \textrm{On event $E_2$, then } \min_{\ell \in [L]} r^2_{{\bm S}_\ell}(\bs{A}, \bs{b}) \leq (1 + \tilde{\epsilon}) r^2_{{\bm S}_{\ell^{\ast\ast}}}(\bs{A}, \bs{b}).
  \end{align}
  We now observe that for any $\eta > 0$, the bound
    \begin{equation}
      \left| \mu(\bs{b}, \bs{S}_{\ell^\ast}) - \mu(\bs{b}, \bs{S}_{\ell^{\ast\ast}})  \right| \leq \eta
    \end{equation}
    implies that
    \begin{align}
      r^2_{\bs{S}_{\ell^{\ast}}}(\bs{A}, \bs{b}) &\leq r^2_{\bs{S}_{\ell^{\ast\ast}}}(\bs{A},\bs{b}) + r^2(\bs{A}, \bs{b}) \big( \eta^2 + 2 \eta \mu(\bs{b}, \bs{S}_{\ell^{\ast\ast}}) \big).
    \end{align}
  Thus, $E_1 \cap E_2$ occurs with probability at least $1 - \delta$, and on this event \eqref{eq:E1} ensures that $\eta$ is given by the right-hand side of \eqref{bound1}. Also, on this event \eqref{eq:E2} implies that $\mu(\bs{b}, \bs{S}_{\ell^{\ast\ast}}) = \tilde{\epsilon}$, i.e., $r^2_{\bs{S}_{\ell^{\ast\ast}}}(\bs{A},\bs{b}) \leq (1 + \tilde{\epsilon}) r^2(\bs{A}, \bs{b})$. Using these expressions in the above inequality, simplifying, and using $(\delta/2)^{1-1/L} \leq \delta/2$ yields the result \eqref{eq:residual-bound-1}.
\end{proof}
We emphasize that the proof above shows how Theorem \ref{thm:optimality} can be used to prove results like Theorem \ref{thm:BFQS-error} for more general sketches.

\subsection{Asymptotic analysis via probabilistic correlation}\label{ssec:theory-asymp}
We provide alternative analysis of Algorithm \ref{alg:BFB} motivated by the following intuition: 
If $\mu(\bs{b}, \bs{S})$ and $\mu(\tilde{\bs{b}}, \bs{S})$ are probabilistically correlated in some sense, then we expect that Algorithm~\ref{alg:BFB} should produce a sketching operator $\bs{S}_{\ell^*}$ that is close to the oracle sketch $\bs{S}_{\ell^{**}}$. We give a technical verification of this intuition below in Theorem \ref{thm:Gaussian}, providing 
an asymptotic lower bound on a certain measure of correlation between the two optimality coefficients when $\bs{S}$ is a Gaussian sketching operator.

\begin{theorem}\label{thm:Gaussian}
	If $\bm S$ is a Gaussian sketch, then 
	\begin{equation}
		\liminf_{m\to\infty}\cor(\mu^2( \bm b, \bm S),\mu^2(\tilde{\bm b}, \bm S))\geq \frac{\|\bm P_{\bm Q_\perp}\bm b_\PP\|_2^2-\sqrt{6}\min\{\|\bm P_{\bm Q_\perp}(\bm b_\PP\pm\tilde{\bm b}_\PP)\|_2\}}{\|\bm P_{\bm Q_\perp}\tilde{\bm b}_\PP\|_2^2},\label{mybound}
	\end{equation}
	where $\bm b_\PP, \tilde{\bm b}_\PP$ are normalized versions of $\bm b$ and $\tilde{\bm b}$, respectively, and the minimum is taken over the two $\pm$ options.
	Moreover, if 
	\begin{equation} \label{eq:phi-kappa}
	    \varphi 
	    \defeq \frac{|\langle \bs{b}, \tilde{\bs{b}} \rangle|}{\| \bs{b} \|_2 \| \tilde{\bs{b}} \|_2} 
	    \geq \frac{\| \bs{P}_{\bs{Q}} \bs{b} \|_2}{ \|\bs{b}\|_2 } 
	    \defeq \kappa,
	\end{equation}
	then we further have that
	\begin{equation} \label{mybound1}
		\liminf_{m\to\infty}\cor(\mu^2( \bm b, \bm S),\mu^2(\tilde{\bm b}, \bm S))\geq (1-\kappa^2)-\frac{\sqrt{12(1-\varphi)}}{(\varphi-\kappa)^2}.
	\end{equation}
\end{theorem}

In Theorem~\ref{thm:Gaussian} we restrict to Gaussian sketches and consider $\cor(\mu^2( \bm b, \bm S),\mu^2(\tilde{\bm b}, \bm S))$ (rather than the more natural quantity $\cor(\mu( \bm b, \bm S),\mu(\tilde{\bm b}, \bm S))$) in order to make analysis tractable.
In general $\cor(\mu(\bm b, \bm S),\mu(\tilde{\bm b}, \bm S))$ and $\cor(\mu^2(\bm b, \bm S),\mu^2(\tilde{\bm b}, \bm S))$ may have significantly different statistical properties.
However, if either of them is close to $1$, then that would indicate a monotonically increasing (although not necessarily linear) relationship between $\mu(\bm b, \bm S)$ and $\mu(\tilde{\bm b}, \bm S)$, and
when such a relationship holds we expect the boosting procedure in Algorithm~\ref{alg:BFQS} to work well.
While we restrict to Gaussian sketches, this probabilistic model is usually a good indicator of how other sketches perform \citep[Remark~8.2]{martinsson2020RandNLA}. I.e., we expect the result to carry over to the random sampling-based sketches (e.g., leverage scores) that we consider.
We verify this numerically in Section~\ref{sec:numerical-experiments}.
	
\begin{remark}\label{rmk:nu-bound-Gaussian}
  The lower bound in \eqref{mybound1} is useful only when the right-hand side is close to $1$, which roughly requires $\varphi$ to be large and $\kappa$ to be small. See Remark \ref{rmk:nu-bound} for how this condition relates to Theorem \ref{thm:optimality}.
\end{remark}

The rest of this section is organized as follows.
Section~\ref{sec:theoretical-analysis-prelims} derives some preliminary technical results.
Section~\ref{sec:theoretical-analysis-optimality-coeff} then proves Theorem~\ref{thm:optimality}.
Section~\ref{sec:sketch-guarantees} provides theoretical guarantees for when various sketches satisfy the $(\varepsilon, \delta)$ pair condition in Definition~\ref{def:epsilon-delta-condition} and discuss how this condition in turn ensures that those sketching operators satisfy the requirements in Theorem~\ref{thm:optimality}.
The proof of Theorem~\ref{thm:Gaussian} is given in Appendix~\ref{app:a}.

\subsection{Preliminary technical results} \label{sec:theoretical-analysis-prelims}

Our first task is to understand how the optimal residual $r(\bs{A}, \bs{b})$ compares to $r_{\bs{S}}(\bs{A}, \bs{b})$. Throughout this section let $\bm Q = \orth(\bm A)$.
\begin{lemma}\label{lemma:diff-of-residuals}
  Given a sketch matrix $\bm S$, assume $\ker(\bs{S}) \cap \range(\bs{A}) = \{\bs{0}\}$, or, equivalently, $\rank(\bs{S} \bs{A}) = \rank(\bs{A})$. Then we have,
  \begin{equation}
  	r^2_{\bm S}(\bm A, \bm b) = r^2(\bm A, \bm b) + \|(\bm S\bm Q)^\dagger\bm S\bm Q_\perp\bm Q^T_\perp\bm b\|_2^2.
  \end{equation}
\end{lemma}
\begin{proof}
Under the assumption $\ker(\bs{S}) \cap \range(\bs{A}) = \{\bs{0}\}$, the sketched least squares problem reproduces elements of $\range(\bs{A})$:
For any $\bs{c} \in \range(\bs{A})$,
\begin{equation}\label{eq:sketch-reproduction}
  \bs{A} (\bm S\bm A)^\dagger\bm S \bs{c} = \bs{c}.
\end{equation}
The solution to the sketched least squares problem \eqref{eq:xast-def} is $(\bm S\bm A)^\dagger\bm S\bm b$. 
Combining this fact with \eqref{eq:r_S} and \eqref{eq:sketch-reproduction} yields
\begin{equation}
  r^2_{\bm S}(\bm A, \bm b) 
  = \|\bm b - \bm A(\bm S\bm A)^\dagger\bm S\bm b\|_2^2
  =  \|\bm b - \bm A(\bm S\bm A)^\dagger\bm S(\bm Q\bm Q^T+\bm Q_\perp\bm Q_\perp^T)\bm b\|_2^2 
  = r^2(\bm A, \bm b) + \|(\bm S\bm Q)^\dagger\bm S\bm Q_\perp\bm Q^T_\perp\bm b\|_2^2.
\end{equation}
\end{proof}

We conclude that $r_{\bm S}(\bm A, \bm b)$ is comparable to $r(\bs{A}, \bs{b})$ if and only if $\|(\bm S\bm Q)^\dagger\bm S\bm Q_\perp\bm Q^T_\perp\bm b\|_2^2$ is small. 

The quantities $\nu$, $\varphi$ and $\kappa$ defined in \eqref{mynu} and \eqref{eq:phi-kappa} are related by the following inequality.
\begin{prop}\label{prop:cor}
	Assume $\varphi\geq\kappa$. Then we have the two inequalities,
        \begin{align}
          \nu&\geq \varphi - \kappa\min\left\{1,\sqrt{2(1-\varphi+\kappa)}\right\}.\label{nu}
        \\\label{nu2}
                \nu&\geq \varphi - (\varphi\tilde{\kappa} + \sqrt{1-\varphi^2})\min\left\{1,\sqrt{2(1-\varphi+\varphi\tilde{\kappa} + \sqrt{1-\varphi^2})}\right\}. 
        \end{align}
        where 
	\begin{equation} \label{kshsy}
		\tilde{\kappa} \defeq \frac{\|\bs{P}_{\bs Q}\tilde{\bm b}\|_2}{\|\tilde{\bm b}\|_2},
	\end{equation}
        measures the relative energy of the \textit{low-fidelity} vector in the range of $\bs{A}$.
\end{prop}

\begin{proof}
  We first prove \eqref{nu}. Since correlation coefficients are scale-invariant, without loss of generality we assume $\|\bm b\|_2 = \|\tilde{\bm b}\|_2 = 1$. 
Write down the orthogonal decomposition of $\bm b$ and $\tilde{\bm b}$ in $\bm Q\oplus\bm Q_\perp$ as follows:
\begin{equation}
\begin{aligned}
	\bm b &= \underbrace{\bm P_{\bm Q}\bm b}_{\bm b_1}+\underbrace{\bm P_{\bm Q_\perp}\bm b}_{\bm b_2}, \\ 
	\tilde{\bm b} &= \underbrace{\bm P_{\bm Q}\tilde{\bm b}}_{\tilde{\bm b}_1}+\underbrace{\bm P_{\bm Q_\perp}\tilde{\bm b}}_{\tilde{\bm b}_2}.
\end{aligned}
\end{equation}
Notice that $\|\bm b_1\|_2^2 + \|\bm b_2\|_2^2 =\|\tilde{\bm b}_1 \|_2^2 + \|\tilde{\bm b}_2 \|_2^2 = 1$.
It follows from the Cauchy--Schwarz inequality and the definitions in \eqref{mynu} and \eqref{eq:phi-kappa} that
\begin{equation} \label{cbb}
  \nu 
  = \frac{|\langle\bm b_2, \tilde{\bm b}_2 \rangle|}{\|\bm b_2\|_2\|\tilde{\bm b}_2 \|_2} 
  \geq |\langle\bm b, \tilde{\bm b}\rangle-\langle\bm b_1, \tilde{\bm b}_1 \rangle| 
  \geq \varphi - \|\bm b_1\|_2\|\tilde{\bm b}_1 \|_2
  = \varphi - \kappa\|\tilde{\bm b}_1 \|_2
  \geq \varphi - \kappa.
\end{equation}
The last inequality can be replaced by a more accurate estimate for $\|\tilde{\bm b}_1 \|_2$:
\begin{equation} \label{bbc}
  	\varphi= |\langle \bm b, \tilde{\bm b}\rangle| = |\langle \bm b_1, \tilde{\bm b}_1 \rangle+\langle \bm b_2, \tilde{\bm b}_2 \rangle|
  	\leq \|\bm b_1\|_2\|\tilde{\bm b}_1 \|_2 + \|\bm b_2\|_2\|\tilde{\bm b}_2 \|_2 
	\leq \kappa + \|\tilde{\bm b}_2 \|_2 = \sqrt{1-\|\tilde{\bm b}_1 \|_2^2}+ \kappa,
\end{equation}
which can be reorganized as
\begin{equation} \label{ccc}
	\|\tilde{\bm b}_1 \|_2\leq\sqrt{1-(\varphi-\kappa)^2} = \sqrt{(1-\varphi+\kappa)(1+\varphi -\kappa)}\leq \sqrt{2(1-\varphi + \kappa)}.
\end{equation}
  Combining \eqref{cbb} and \eqref{ccc} finishes the proof of \eqref{nu}.

  To show \eqref{nu2}, we again assume $\|\bm b\|_2 = \|\tilde{\bm b}\|_2 = 1$, so that,
	\begin{equation}
		\kappa 
		= \|\bs{P}_{\bs Q}\bs{b}\|_2
		= \|\bs{P}_{\bs Q}(\bm P_{\tilde{\bm b}} \bm b + \bm b - \bm P_{\tilde{\bm b}} \bm b)\|_2
		\leq \varphi\|\bs P_{\bs Q}\tilde{\bm b}\|_2 + \|\bs b-\bm P_{\tilde{\bs b}}\bs b\|_2
		= \varphi\tilde{\kappa} +\sqrt{1-\varphi^2}.
	\end{equation}
  Plugging this into \eqref{nu} and noting that the right-hand side of \eqref{nu} is decreasing in $\kappa$ yields \eqref{nu2}.
\end{proof}

The main appeal of \eqref{nu2} is that the quantity $\tilde{\kappa}$ involves only low-fidelity data, and hence can be estimated. I.e., \eqref{nu2} gives a more practically computable lower bound for $\nu$, involving one quantity $\tilde{\kappa}$ that depends only on low-fidelity data $\tilde{\bm b}$, and the correlation $\varphi$ between $\bm b$ and $\tilde{\bm b}$.
\begin{remark}\label{rmk:nu-bound}
  Recall that our main convergence result, Theorem \ref{thm:optimality}, has more attractive bounds when $\nu$ is large. By \eqref{nu}, $\nu$ is large if $\varphi\approx 1$ and $\varphi\gg\kappa$, which coincides with sufficient conditions to ensure attractive bounds in \eqref{mybound1} in Theorem \ref{thm:Gaussian}. (Cf. Remark \ref{rmk:nu-bound-Gaussian}.) Thus, $\varphi \gg \kappa$ is a unifying condition under which both of our main theoretical results, Theorem \ref{thm:optimality} and Theorem \ref{thm:Gaussian}, provide useful bounds. The condition $\varphi \gg \kappa$ means that the correlation between $\bs{b}$ and $\tilde{\bs{b}}$ is high and strongly dominates the relative energy of $\bs{b}$ in $\range(\bs{A})$.
	This condition may seem counterintuitive as it requires the high-fidelity solution to have a relatively large residual.
	Since $\mu$ is defined relative to $r(\bm A, \bm b)$, a small $r_{\bm S_{\ell^*}}(\bm A, \bm b)$ may still result in a large $\mu (\bm b, \bm S_{\ell^*})$ even if $r_{\bm S_{\ell^*}}(\bm A, \bm b)$ is small but relatively large compared to $r(\bm A, \bm b)$.
\end{remark}

\subsection{Proof of Theorem~\ref{thm:optimality}} \label{sec:theoretical-analysis-optimality-coeff}

We first consider the case $\cor(\bm P_{\bm Q_\perp}\bm b, \bm P_{\bm Q_\perp}\tilde{\bm b})\geq 0$.
Fixing $\ell\in [L]$, $\bm S = \bm S_\ell$, consider the event $E$ of probability at least $1 - \delta/L$ where the rank condition in \eqref{mypair} holds. On this event, this rank condition with Lemma \ref{lemma:diff-of-residuals} implies that,
\begin{align*}
  r^2_{\bm S}(\bm A, \bm b) - r^2(\bm A, \bm b) = \|(\bm S\bm Q)^\dagger\bm S\bm Q_\perp\bm Q^T_\perp\bm b\|_2^2,
\end{align*}
allowing us to directly estimate the difference between $\mu(\bm b, \bm S)$ and $\mu(\tilde{\bm b}, \bm S)$ as follows:
\begin{equation}
\begin{aligned} \label{erf}
	|\mu(\bm b, \bm S) - \mu(\tilde{\bm b}, \bm S)|
	&=\left|\frac{\|(\bm S\bm Q)^\dagger\bm S\bm Q_\perp\bm Q^T_\perp\bm b\|_2}{\|\bm Q_\perp\bm Q^T_\perp\bm b\|_2}-\frac{\|(\bm S\bm Q)^\dagger\bm S\bm Q_\perp\bm Q^T_\perp\tilde{\bm b}\|_2}{\|\bm Q_\perp\bm Q^T_\perp\tilde{\bm b}\|_2}\right|\\
	&\leq \left\|(\bm S\bm Q)^\dagger\bm S\left((\bm P_{\bm Q_\perp}\bm b)_\PP-(\bm P_{\bm Q_\perp}\tilde{\bm b})_\PP\right)\right\|_2\\
	&= \|(\bm P_{\bm Q_\perp}\bm b)_\PP-(\bm P_{\bm Q_\perp}\tilde{\bm b})_\PP\|_2\|(\bm S\bm Q)^\dagger\bm S\bm h\|_2\\
	& =  \sqrt{\|(\bm P_{\bm Q_\perp}\bm b)_\PP\|^2_2+\|(\bm P_{\bm Q_\perp}\tilde{\bm b})_\PP\|^2_2-2\langle(\bm P_{\bm Q_\perp}\bm b)_\PP, (\bm P_{\bm Q_\perp}\tilde{\bm b})_\PP\rangle} \; \|(\bm S\bm Q)^\dagger\bm S\bm h\|_2\\
	& = \sqrt{2-2\nu}\cdot\|(\bm S\bm Q)^\dagger\bm S\bm h\|_2\\
	& = \sqrt{2-2\nu}\cdot\|\bm Q(\bm S\bm Q)^\dagger\bm S\bm h\|_2,
\end{aligned}
\end{equation}
where the first inequality follows from the reverse triangle inequality, the second to last equality follows \eqref{mynu}, and the final equality follows from unitary invariance of the operator norm.
The case $\cor(\bm P_{\bm Q_\perp}\bm b, \bm P_{\bm Q_\perp}\tilde{\bm b})< 0$ can be treated similarly by noting that the inequality on the second line of \eqref{erf} still holds if the minus sign on the right-hand side is changed to a plus sign.
The rest of the computation is then done similarly to the case with non-negative correlation.

Note that $(\bm S\bm Q)^\dagger\bm S\bm h$ is the $\bm S$-sketched least squares solution to $\min_{\bs{x}} \| \bm Q\bm x - \bm h \|_2$. 
Also, note that $\bs{h} \in \range(\bs{Q}_\perp)$.
Using the residual bound in \eqref{mypair},
the following also holds on our probabilistic event $E$:
\begin{equation}
	\|\bm Q(\bm S\bm Q)^\dagger\bm S\bm h\|_2^2 + \|\bm h\|_2^2 = \|\bm Q(\bm S\bm Q)^\dagger\bm S\bm h - \bm h\|_2^2\leq (1+\e)^2\min_{\bm x\in\R^d}\|\bm Q\bm x-\bm h\|^2_2 = (1+\e)^2\|\bm h\|_2^2.
\end{equation}
Rearranging terms and noting $\|\bm h\|_2 = 1$ yields $\|\bm Q(\bm S\bm Q)^\dagger\bm S\bm h\|\leq \sqrt{3\e}$, which is substituted into \eqref{erf}, implying that on an event $E$ with probability at least $1 - \delta/L$, we have
\begin{equation}
	|\mu(\bm b, \bm S) - \mu(\tilde{\bm b}, \bm S)|\leq \sqrt{6(1-\nu)\e}.
\end{equation}
Taking a union bound over $\ell\in [L]$ yields that, with probability at least $1-\delta$,
\begin{equation}
	\max_{\ell\in [L]}|\mu(\bm b, \bm S_\ell) - \mu(\tilde{\bm b}, \bm S_\ell)|\leq \sqrt{6(1-\nu)\e}.\label{goodevent}
\end{equation}
Conditioning on the probabilistic event in \eqref{goodevent} and using the definition of $\ell^*$ and $\ell^{**}$ finishes the proof: 
\begin{equation}
	\mu(\bm b, \bm S_{\ell^*}) \leq \mu(\tilde{\bm b}, \bm S_{\ell^*}) + \sqrt{6(1-\nu)\e}\leq \mu(\tilde{\bm b}, \bm S_{\ell^{**}}) + \sqrt{6(1-\nu)\e} \leq \mu(\bm b, \bm S_{\ell^{**}}) + 2\sqrt{6(1-\nu)\e}.
\end{equation}

\subsection{Achieving the \texorpdfstring{$(\e,\delta)$}{(epsilon,delta)} pair condition} \label{sec:sketch-guarantees}

We next show that, for a variety of random sketches of interest, the $(\e, \frac{\delta}{L})$ pair condition for $(\bm Q, \bm h)$ in Theorem \ref{thm:optimality} holds for sufficiently large $m$. 
We begin with a lemma that gives a sufficient condition for verification of the $(\e, \frac{\delta}{L})$ pair condition for $(\bm Q, \bm h)$, which can be deduced as a special case from \cite[Lemma 1]{drineas2011FasterLeast}:

\begin{lemma}[\citet{drineas2011FasterLeast}]\label{lem:Drineas2011}
Let $\bs{Q}$ and $\bs{h}$ be defined as in Theorem~\ref{thm:optimality}. 
The distribution of $\bm S$ is an $(\e, \frac{\delta}{L})$ pair for $(\bm Q, \bm h)$ if the following two conditions hold simultaneously with probability at least $1-\delta/L$:
\begin{equation} \label{mycond}
	\sigma_{\min}^2 (\bm S\bm Q) \geq \frac{\sqrt{2}}{2} 
	\;\;\;\; \text{and}	\;\;\;\; 
	\|\bm Q^T\bm S^T\bm S\bm h\|_2^2\leq\frac{\e}{2},
\end{equation}
where $\sigma_{\min}(\cdot )$ denotes the smallest singular value of a matrix. 
\end{lemma}

When the conditions in Lemma \ref{lem:Drineas2011} hold, one can directly bound \eqref{erf} using the submultiplicativity of operator norms instead of resorting to an $(\e, \delta)$ argument as in the proof of Theorem \ref{thm:optimality}, although the latter is more general. 
Theorem~\ref{thm:sketches} presents constructive strategies for generating sketch distributions -- based on sub-Gaussian random variables and leverage scores -- that achieve appropriate $(\e, \delta)$ pair conditions. 
We recall that a random variable $X$ is called sub-Gaussian if, for some $K > 0$ we have $\E \exp(X^2/K^2) \leq 2$
\citep[Def.~2.5.6]{vershynin2018HighdimensionalProbability}. 
The sub-Gaussian norm of $X$ is defined as 
$\Vert X\Vert_{\psi_2}\defeq\inf\left\{K> 0\;:\; \mathbb{E}\exp(X^2/K^2)\leq 2\right\}$
~\citep{vershynin2018HighdimensionalProbability}. A proof of Theorem~\ref{thm:sketches} is give in Appendix~\ref{app:a0}.
Variants of these results have appeared previously in the literature \cite{drineas2006SamplingAlgorithms, drineas2008RelativeerrorCUR, drineas2011FasterLeast, larsen2020PracticalLeverageBased}.

\begin{theorem}\label{thm:sketches}
Let $\bs{Q}$ and $\bs{h}$ be defined as in Theorem~\ref{thm:optimality}. 
Write $\bm Q$ and $\bm S^T$ as column vectors: 
\begin{equation}
	\bm Q = [\bm q_1 , \cdots, \bm q_d], \;\;\;\;  \bm S^T = [\bm s_1, \cdots, \bm s_m], 
\end{equation}
and denote by $q_{ij}\defeq \bm q_i(j)$ and $h_{j}\defeq \bm h(j)$ the $j$-th component of $\bm q_i$ and $\bm h$, respectively. 
\begin{enumerate}
\item Suppose $\bm S\in\R^{m\times N}$ is a dense sketch whose entries are i.i.d.\ sub-Gaussian random variables with mean 0 and variance $1/m$.
Assume the sub-Gaussian norm of each entry of $\sqrt{m}\bm S$ is bounded by $K\geq 1$. 
Then the distribution of $\bm S$ is an $(\e, \frac{\delta}{L})$ pair for $(\bm Q, \bm h)$ if 
\begin{equation}
	m \geq \frac{CK^4}{\e}d\log \left(\frac{4dL}{\delta}\right),\label{subgau}
\end{equation}
where $C$ is an absolute constant.  
\item\label{itm:lev-score} Suppose $\bm S\in\R^{m\times N}$ is a row sketch based on the leverage scores of $\bm A$, and $0<\e, \delta<1/2$; see Equation~\eqref{eq:lev-score-sketch}. Then the distribution of $\bm S$ is an $(\e, \frac{\delta}{L})$ pair for $(\bm Q, \bm h)$ if 
\begin{equation}
	m\geq \max\left\{35d\log\left(\frac{4dL}{\delta}\right), \frac{2dL}{\e\delta}\right\}.\label{lop}
\end{equation}
Moreover, if
\begin{equation}
 	\max_{i\in [d]}\max_{j\in [N]: \ell_j>0}\frac{d| q_{ij} h_j|}{\ell_j}\leq C, \;\;\;\; \ell_j = \sum_{k\in [d]}q^2_{kj}\label{pang}
\end{equation}
for some constant $C>0$, then the distribution of $\bm S$ is an $(\e, \frac{\delta}{L})$ pair for $(\bm Q, \bm h)$ if
\begin{equation}
	m\geq \max\left\{35, \frac{4C^2}{\e}\right\}d\log \left(\frac{4dL}{\delta}\right).\label{lop_more}
\end{equation}
\end{enumerate}
\end{theorem}
The scalar $\ell_j$ in \eqref{pang} is the leverage score associated to row $j$ of $\bs{A}$, and $(\ell_j)_{j \in [N]}$ defines a (discrete) probability distribution over the row indices $[N]$ of $\bs{A}$; see \eqref{eq:lev-score-sketch}. 
\begin{remark}
When $\bm Q$ is incoherent, i.e., when its leverage scores satisfy $\ell_i = \mathcal{O}(d/N)$, the entries $q_{ij}$ satisfy $q_{ij} = \mathcal{O}(1/\sqrt{N})$. 
For any $\bm h$ such that $\max_{j\in [N]}|h_j| \lesssim\mathcal{O}(1/\sqrt{N})$, the condition in \eqref{pang} is satisfied with $C = \mathcal{O}(1)$:  
\begin{align*}
\max_{i\in [d]}\max_{j\in [N]: \ell_j>0}\frac{d|q_{ij} h_j|}{\ell_j}\lesssim \frac{d\cdot\frac{1}{\sqrt{N}}\cdot\frac{1}{\sqrt{N}}}{\frac{d}{N}} = 1.
\end{align*} 
\end{remark}

\begin{remark}
    As noted in Section~\ref{sssec:lev_vol_sampling}, leveraged volume sampling requires $m \gtrsim d \log(d / \delta) + d / (\varepsilon \delta)$ samples to satisfy the $(\varepsilon,\delta)$ pair condition.
    This result appears in Corollary 10 of \cite{derezinski2018LeveragedVolume}.
\end{remark}

\section{Numerical experiments} \label{sec:numerical-experiments}

In this section we illustrate various aspects of the BFB approach using both manufactured data as well as data obtained from PDE solutions. The codes used to generate the results of this section are available from the GitHub repository {\tt \href{https://github.com/CU-UQ/BF-Boosted-Quadrature-Sampling}{{https://github.com/CU-UQ/BF-Boosted-Quadrature-Sampling}}}. 

\subsection{Verification of theoretical results on synthetic data} \label{sec:synthetic-experiments}
 
We first verify the theoretical results in Theorems \ref{thm:optimality} and \ref{thm:Gaussian}. 
We do this by simulating different values for $\bm S$, $\bm b$ and $\tilde{\bm b}$. 
We generate a design matrix $\bm A\in\R^{1000\times 50}$ (i.e., $N = 1000$ and $d=50$) with i.i.d.\ standard normal entries and fix it in the rest of the simulations. 
For sketching matrices $\bm S$, we choose the embedding dimension to be $m = 100$ and consider both the Gaussian and leverage score sampling sketches. 
We generate multiple different versions of the vectors $\bs{b}$ and $\tilde{\bs{b}}$ that correspond to different values of $\kappa$ and $\varphi$.
Recall that these parameters control how much of $\bs{b}$ is in the range of $\bs{A}$ and the absolute value of the correlation between $\bs{b}$ and $\tilde{\bs{b}}$, respectively.
The vectors are generated via
\begin{equation}
\begin{aligned}
	\bs{b} &= \kappa \bs{Q} \bs{z}_1 + \sqrt{1 - \kappa^2} \bs{Q}_\perp \bs{z}_2, \\
	\tilde{\bs{b}} &= \varphi \bs{b} + \sqrt{1 - \varphi^2} \bs{b}_\perp \bs{z}_3,
\end{aligned}
\end{equation}
where $\bs{Q} = \orth(\bs{A})$, and $\bs{z}_1 \in \mathbb{R}^{d-1}$, $\bs{z}_2 \in \mathbb{R}^{N-d-1}$ and $\bs{z}_3 \in \mathbb{R}^{N-2}$ are generated by normalizing random vectors of appropriate length whose entries are i.i.d.\ standard normal.
In the experiment, the vectors $\bs{z}_1, \bs{z}_2, \bs{z}_3$ are drawn once and then kept fixed for the different choices of $\kappa$ and $\varphi$.

To check the upper bound in Theorem~\ref{thm:optimality}, we generate $\bs{b}$ and $\tilde{\bs{b}}$ using 9 equi-spaced values for $\varphi$ and $\kappa$ between 0 and 1, which will provide 81 plots for each sketching strategy. 
We use a sequence of $L=10$ independent sketching operators in our BFB approach.
After computing values of $\nu$ for every case, we evaluate the optimality coefficient difference $\mu(\bm b, \bm S_{\ell^*}) - \mu(\bm b, \bm S_{\ell^{**}})$. 
 Figure~\ref{fig:thm33bound} illustrates the relation between $\mu(\bm b, \bm S_{\ell^*}) - \mu(\bm b, \bm S_{\ell^{**}})$ and the bound $2\sqrt{6(1-\nu)\e}$. 
 Due to the unknown constants in \eqref{subgau} and \eqref{lop}, an exact value of $\e$ corresponding to $m=100$ is unavailable.
 Instead, we choose $\e$ to be 0.01 heuristically.
 We chose this particular value of $\e$ since it illustrates how the green curve's shape, which is independent with the scalar $\e$, separates most of the scatter plots from the rest of the area. 
The result shows our purposed BFB bound in 
 Theorem \ref{thm:optimality} is effective and non-vacuous for both Gaussian and leverage score sketchings. 
 It is noticeable that all the dots out of our proposed bound (green) are leverage score sketch spots (blue). 
 The reason is because we set $m=100$ for both sketch strategies, while leverage score sketch requires a higher $m$ to satisfy the  $(\e,\delta)$ pair condition, which leads to a higher deviation in $\mu$ with fixed $m$; see details in Theorem~\ref{thm:sketches}. 

\begin{figure}
    \centering
    \includegraphics[width=0.6\textwidth]{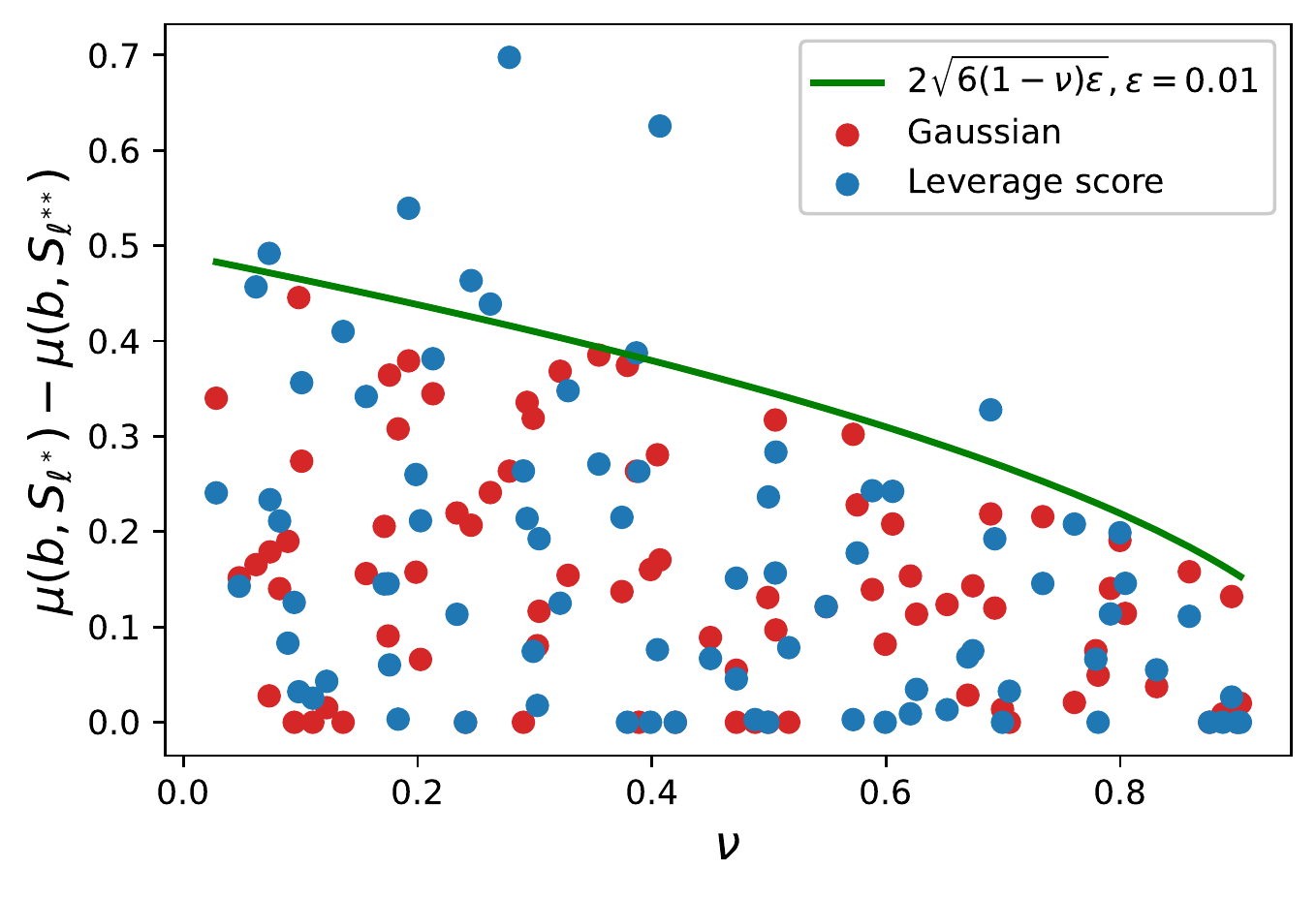}
    \caption{Scatter plots of $\mu(\bm b, \bm S_{\ell^*}) - \mu(\bm b, \bm S_{\ell^{**}})$ based on given values of $\nu$ for Gaussian sketch (red) and leverage score sketch (blue). The green curve is the bound we provide in Theorem~\ref{thm:optimality} with $\e=0.01$.}
    \label{fig:thm33bound}
\end{figure}

\begin{table}[ht]
	\centering
	\caption{Empirical correlation between $\mu^2(\bm A, \bm b)$ and $\mu^2(\bm A, \tilde{\bm b})$ for four different parameters setups and two different sketch types.
	}
	\begin{tabular}{S[table-format=1.2]S[table-format=1.2]cS[table-format=1.2]}
	\toprule
	$\kappa$ & $\varphi$ & Sketch type & {Correlation}  \\
	\midrule
	0.2 	& 0.3 	& Gaussian 			& 0.21 \\
	0.2 	& 0.95 	& Gaussian 			& 0.88 \\
	0.95 	& 0.3 	& Gaussian 			& 0.17 \\
	0.95 	& 0.95 	& Gaussian 			& 0.48 \\
	\midrule
	0.2 	& 0.3 	& Leverage score 	& 0.19 \\
	0.2 	& 0.95 	& Leverage score 	& 0.91 \\
	0.95 	& 0.3 	& Leverage score 	& 0.08 \\
	0.95 	& 0.95 	& Leverage score 	& 0.56 \\
	\bottomrule
	\end{tabular}
	\label{hghg}
\end{table}

\begin{figure}[htb]
	\begin{center}
	\includegraphics[width=0.99\textwidth]{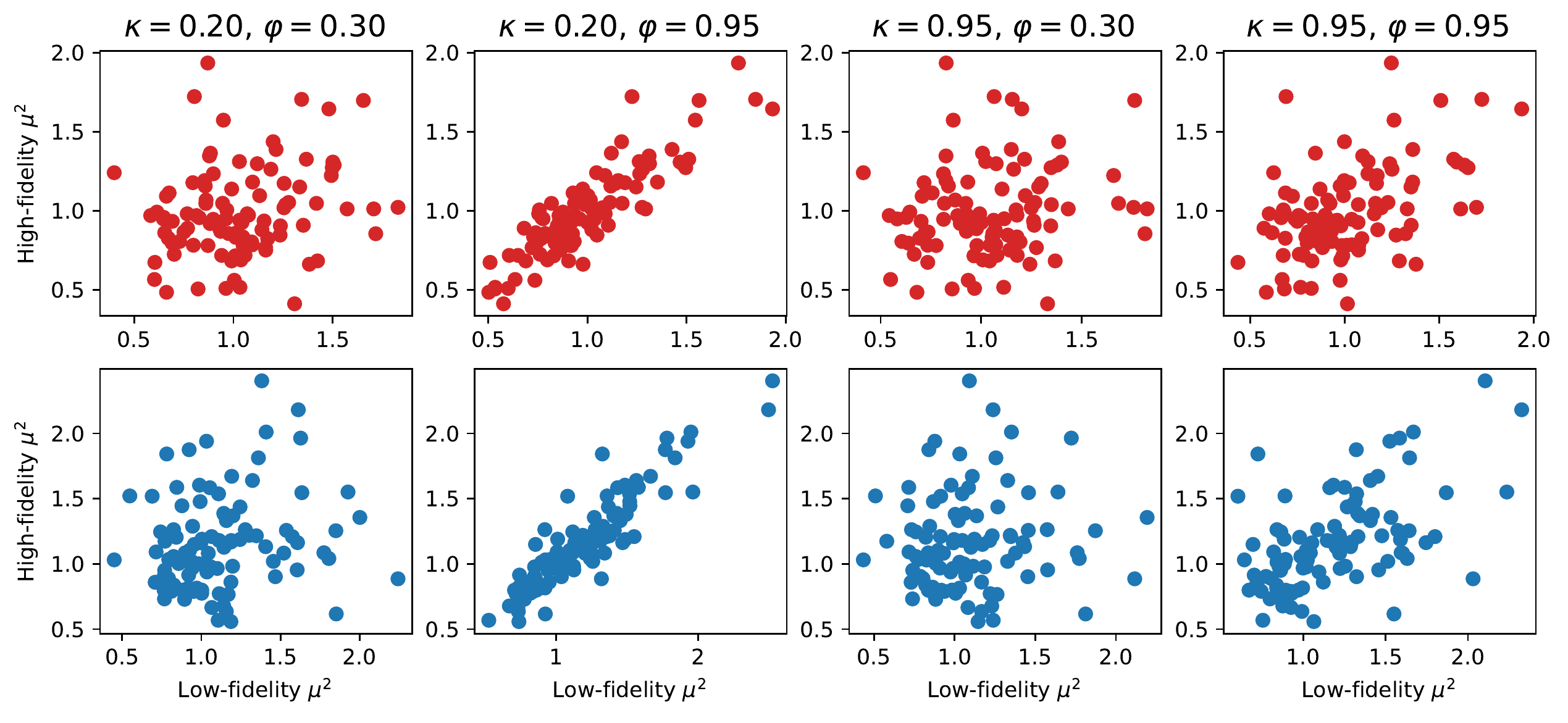}
  \end{center}
  		\caption{Scatter plots of the square of the optimality coefficient for high- and low-fidelity data for each of $100$ different sketches.
			Each point is equal to $(\mu^2(\tilde{\bm b}, \bm S), \mu^2(\bm b, \bm S))$ for one realization of the sketch $\bs{S}$.
			The top and bottom panels correspond to the sketches constructed using Gaussian and leverage score sampling sketches, respectively.\label{fig:beta}}	
\end{figure}

To further validate our theoretical results in Theorem~\ref{thm:Gaussian}, we consider four combinations of $\kappa$ and $\varphi$ as listed in Table~\ref{hghg}.
For both the Gaussian and leverage score sketches we draw 100 sketches randomly.
The same set of sketches are used for each pair of the vectors $\bs{b}$ and $\tilde{\bs{b}}$.
Figure~\ref{fig:beta} shows scatter plots of the squared optimality coefficients for the four different pairs of $\bs{b}$ and $\tilde{\bs{b}}$ and two different sketch types.

Table~\ref{hghg} provides the estimated correlations between $\mu^2(\bm b, \bm S)$ and $\mu^2(\tilde{\bm b}, \bm S)$ for each of the eight setups based on the data points in Figure~\ref{fig:beta}.
For both sketches, a small value of $\kappa$ and a large value of $\varphi$ together yield the highest positive correlation between $\mu^2(\bm b, \bm S)$ and $\mu^2(\tilde{\bm b}, \bm S)$.
In this case, the sketch that attains the smallest residual on the low-fidelity data also attains a near-minimal residual on the high-fidelity data. 
This is indicative of the desired sketch transferability between the low- and high-fidelity regression problems. 
These observations are consistent with the upper bound in \eqref{bound1} and the lower bound in \eqref{mybound1}, supporting the idea of BFB.

\subsection{Experiments on PDE datasets} \label{sec:experiments-pdf-datasets}

In this section we verify the accuracy of Algorithm \ref{alg:BFQS} on two PDE problems:
Thermally-driven cavity fluid flow (Section~\ref{sec:cavity-fluid-flow}) and simulation of a composite beam  (Section~\ref{sec:composite-beam}).
In doing so, we consider three random sketching strategies based on uniform, leverage score (Section~\ref{sssec:lev_score_sampling}), and leveraged volume (Section~\ref{sssec:lev_vol_sampling}) sampling.
As a baseline, we also present results based on deterministic sketching via column-pivoted QR decomposition (Section~\ref{sssec:qr_sampling}).

In both experiments, the high-fidelity solution operator takes uniformly distributed inputs $\bs{p} \in [-1,1]^q$.
We therefore consider approximations of the form in \eqref{eq:function-approximation} with $\psi_j : [-1, 1]^q \mapsto \Rb$ chosen to be products of $q$ univariate (normalized) Legendre polynomials. Specifically, let $\bm j=(j_1,\dots,j_q)$, $j_k\in\mathbb{N}\cup\{0\}$, be a vector of non-negative indices and $\psi_{j_k}(p_k)$ denote the Legendre polynomial of degree $j_k$ in $p_k$ such that $\mathbb{E}[\psi_{j_k}^2(p_k)]=1$. The multivariate Legendre polynomials are given by
\begin{equation} \label{eq:polynomial-product}
	\psi_{\bs j}(\bs{p}) = \prod_{k=1}^q \psi_{j_k}(p_k).
\end{equation}
The set of polynomials $\{\psi_{\bm j}\}$
is chosen so that it spans either a total degree or hyperbolic cross space. In the former case this means all polynomials satisfying $\sum_{k=1}^q j_k\le \zeta$, while in the latter case $\bm j$ is limited to multi-indices with $\prod_{k = 1}^q (j_k + 1) \leq \zeta+1$, for some predefined $\zeta\in\mathbb{N}\cup\{0\}$. 

In order to construct a design matrix $\bs{A}$ as in \eqref{myleastsquares}, and data vectors $\bs{b}$ and $\tilde{\bs{b}}$ in \eqref{myleastsquares} and \eqref{eq:vector-bt-construction}, respectively, we also need to choose pairs of quadrature points and weights $(\bs{p}_n, w_n)_{n\in[N]}$. While both deterministic and random rules are possible, we here choose these quantities to be deterministic and of the form
\begin{equation} \label{eq:product-Gauss-Legendre}
\begin{aligned}
	\bs{p}_n &= (p_{1,n_1}, p_{2,n_2}, \ldots, p_{q,n_q}), \\
	w_n &= \prod_{k=1}^{q} w_{k,{n_k}},
\end{aligned}
\end{equation}
where each sequence $(p_{k,n_k}, w_{k,n_k})_{n_k\in [N_k]}$ consists of node-weight pairs in the $N_k$-point Gauss--Legendre quadrature on $[-1,1]$.
The resulting sequence $(\bs{p}_n, w_n)_{n\in[N]}$ contains $N = \prod_{k=1}^q N_k$ pairs.
When $\bs{A}$ is constructed in this fashion, it is possible to sample rows of that matrix according to the exact leverage score using the efficient method by \citep{malik2022FastAlgorithms}.
Please see Appendix~\ref{sec:lev-score-sampling-alg} for details on how this is done.

To measure the final performance, we use the relative error defined as
\begin{equation} \label{eqn:rel_err}
   E \defeq \frac{\| \bm A \hat{\bs{x}}_\bfqs - \bm b \|_2}{\| \bm b \|_2},
\end{equation}
where $\hat{\bm x}_\bfqs$ is the output from Algorithm~\ref{alg:BFB}.

\subsubsection{Cavity fluid flow} 
\label{sec:cavity-fluid-flow}

Here we consider the case of temperature-driven fluid flow in a 2D cavity \citep{bc15,phd14,hd15,hd15_2,hd18}, with the quantity of interest being the heat flux averaged along the hot wall as Figure \ref{fig:cavity_scheme_color} shows. 
The wall on the left hand side is the hot wall with random temperature $T_h$, and the cold wall at the right hand side has temperature $T_c<T_h$. $\bar{T}_c$ is the constant mean of $T_c$. The horizontal walls are adiabatic. The reference temperature and the temperature difference are given by $T_{\textup{ref}}=(T_h+ \bar{T}_c)/2$ and $\Updelta T_{\textup{ref}}=T_h-\bar{T}_c$, respectively. 
The normalized governing equations are given by
\begin{equation}\label{eqn:cavity}
\begin{aligned}
	&\frac{\partial \u}{\partial t} + \u\cdot\nabla \u=-\nabla p +\frac{\text{Pr}}{\sqrt{\text{Ra}}}\nabla^2 \u+\text{Pr}\Theta\textbf{e}_y,\\
	&\nabla\cdot\u=0,\\
	&\frac{\partial \Theta}{\partial t}+\nabla\cdot(\u\Theta)=\frac{1}{\sqrt{\text{Ra}}}\nabla^2\Theta,
\end{aligned}
\end{equation}
 where $\textbf{e}_y$ is the unit vector $(0,1)$, $\u=(u,v)$ is the velocity vector field, $\Theta=(T-T_{\textup{ref}})/\Updelta T_{\textup{ref}}$ is normalized temperature, $p$ is pressure, and $t$ is time. We assume no-slip boundary conditions on the walls. 
 The dimensionless Prandtl and Rayleigh numbers are defined as $\text{Pr}=\nu_\text{visc}/\alpha$ and $\text{Ra}=g\tau\Updelta T_{\textup{ref}}W^3/(\nu_\text{visc}\alpha)$, respectively, 
where $W$ is the width of the cavity, $g$ is gravitational acceleration, $\nu_\text{visc}$ is kinematic viscosity, $\alpha$ is thermal diffusivity, and $\tau$ is the coefficient of thermal expansion. We set $g=10$, $W=1$, $\tau=0.5$, $\Updelta T_{\textup{ref}}=100$, $\text{Ra}=10^6$, and $\text{Pr}=0.71$. On the cold wall, we apply a temperature distribution with stochastic fluctuations as 
\begin{equation}
    T(x=1,y)=\bar{T}_c+\sigma_T\sum_{i=1}^{q} \sqrt{\lambda_i}\phi_i(y)\mu_i,
\end{equation}
where $\bar{T}_c=100$ is a constant, 
$\{\lambda_i\}_{i\in[q]}$ and $\{\phi_i(y)\}_{i\in[q]}$ are the $q$ largest eigenvalues and corresponding eigenfunctions of the kernel $k(y_1,y_2)=\exp(-|y_1-y_2|/0.15)$, and each $\mu_i \overset{\text{i.i.d.}}{\sim} U[-1, 1]$.
We let $q=2$ (though in general, this does not need to match the physical dimension) and $\sigma_T=2$. 
The vector $\bs{p} = (\mu_1, \, \mu_2)$ is the uncertain input of the model.

\begin{figure}[htb]
    \centering
    \includegraphics[width=0.4\textwidth]{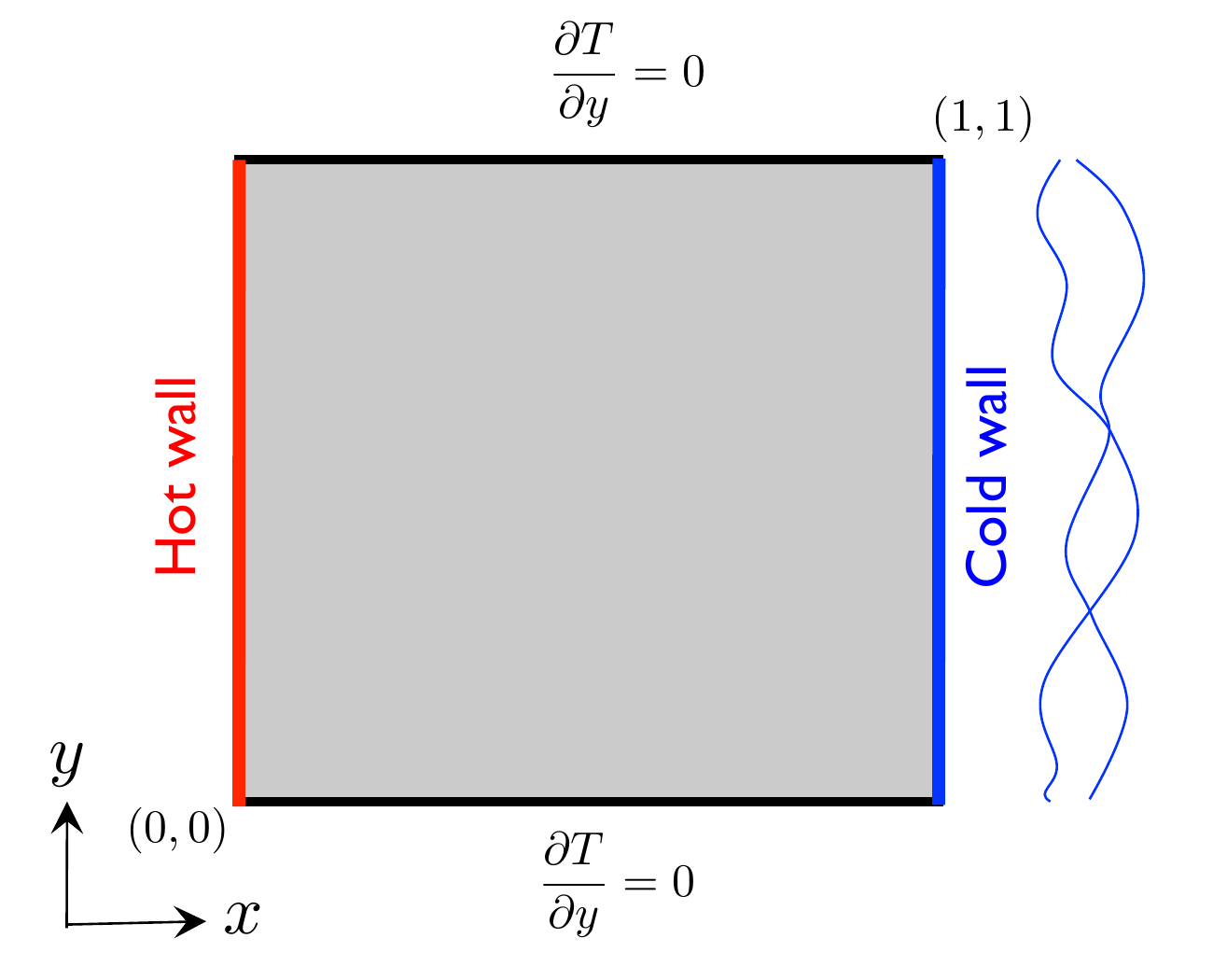}
    \caption{A figure of the temperature driven cavity flow problem, reproduced from Figure~5 of \cite{fdki17}.}
    \label{fig:cavity_scheme_color}
\end{figure}

In order to solve \eqref{eqn:cavity} we use the finite volume method with two different grid resolutions:
a finer grid of size $128 \times 128$ to produce the high-fidelity solution and a coarser grid of size $16 \times 16$ to produce the low-fidelity solution. For our surrogate model, we choose the basis set $\{\psi_j\}_{j\in[d]}$ based on the total degree and hyperbolic cross spaces of maximum order $\zeta=4$. The corresponding spaces have $d=15$ and $d=10$ basis functions, respectively.
The quadrature pairs $(\bs{p}_n, w_n)$ used to construct $\bs{A}$, $\bs{b}$, and $\tilde{\bs{b}}$ are defined as in \eqref{eq:product-Gauss-Legendre} and are based on the nodes and weights from a 10-point Gauss--Legendre rule, i.e., $N_1 = N_2 = 10$.

We first repeat the test we ran on synthetic data in Section~\ref{sec:synthetic-experiments}.
Figure~\ref{fig:corr_plots} shows the scatter plots of $(\mu^2(\tilde{\bm b}, \bm S), \mu^2(\bm b, \bm S))$ for the two different polynomial spaces and three different random sampling approaches.
Each plot is based on 100 sketches with $m=30$ and $m=20$ samples used for the total degree and hyperbolic cross spaces, respectively. Table~\ref{table:cavity-flow} presents the correlation coefficients between $\mu^2(\bm b,\bm S)$ and $\mu^2(\tilde{\bm b},\bm S)$ based on the points in Figure~\ref{fig:corr_plots}.
There is a discrepancy between the correlation observed for the total degree and hyperbolic cross spaces.
One possible explanation for this is that a greater portion of $\bs{b}$ is in the range of $\bs{A}$ for the total degree space than for the hyperbolic cross space, i.e., $\kappa$ (see \eqref{eq:phi-kappa}) is larger for the former space.
Theorem~\ref{thm:Gaussian} indicates that a larger $\kappa$ should be associated with lower correlation.

\begin{figure}[H]
    \centering
    \includegraphics[width=1.0\textwidth]{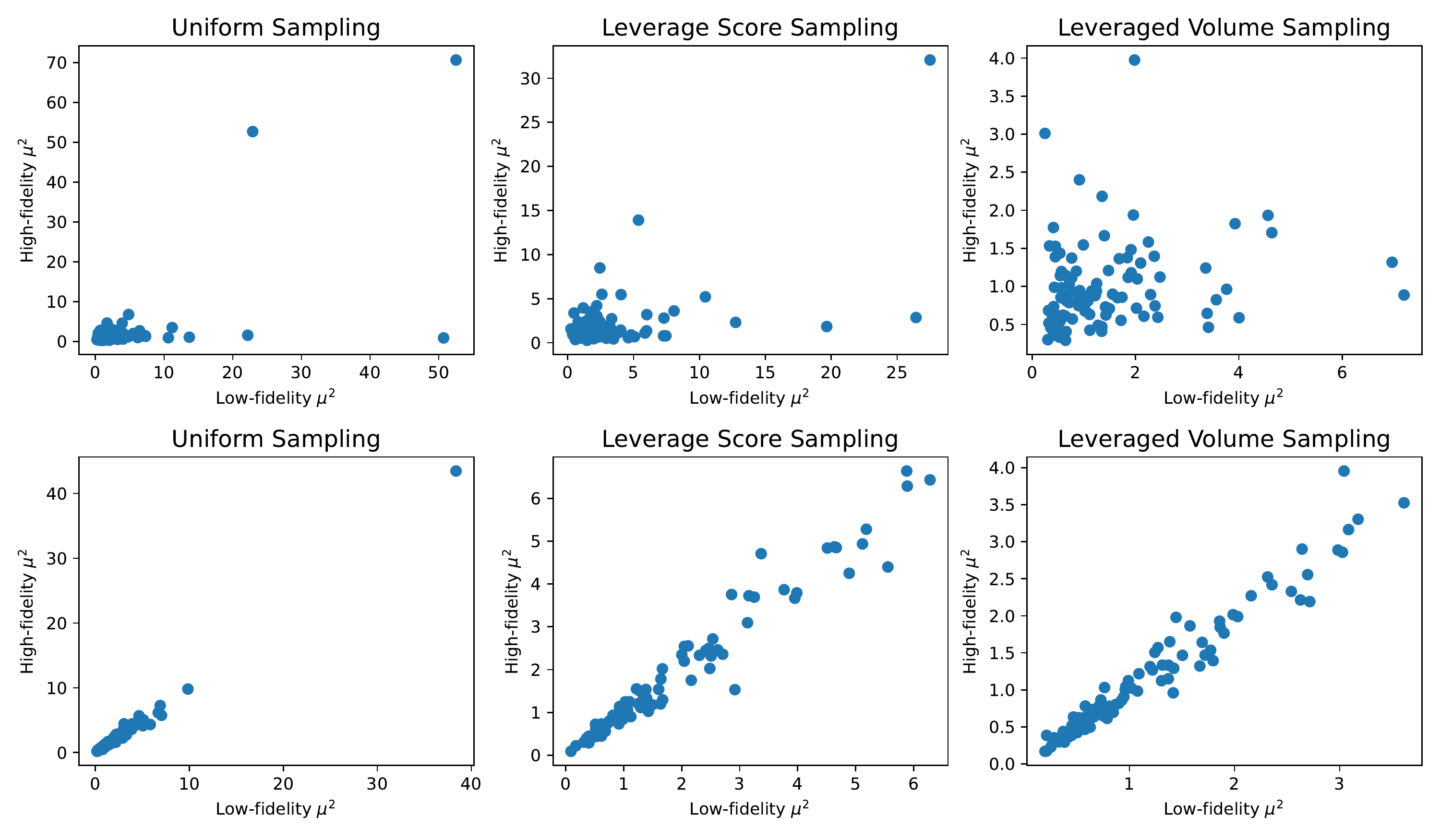}
    \caption{Scatter plots of the square of the optimality coefficient for high- and low-fidelity data from the cavity fluid flow problem for different polynomial spaces (top: total degree; bottom: hyperbolic cross) and types of sampling.
    	Each point is equal to $(\mu^2(\tilde{\bm b}, \bm S), \mu^2(\bm b, \bm S))$ for one realization of the sketch $\bs{S}$, and each subplot contains 100 points (i.e., is based on 100 sketch realizations).
    	For the total degree space $m=30$ samples are used and for the hyperbolic cross space $m=20$ samples are used.
    	The corresponding correlation coefficients are presented in Table \ref{table:cavity-flow}.}
    \label{fig:corr_plots}
\end{figure}

Next, we run Algorithm~\ref{alg:BFB} with $L=10$ sketches and the number of samples $m=1.2d$ and $m=2d$.
Figure~\ref{fig:cavity_flow} shows the relative error $E$ in (\ref{eqn:rel_err}) from running the algorithm 1000 times for each of the different choices of polynomial space, sketch size $m$, and random sampling approach. We observe that in all cases the BFB approach improves the error as compared to the non-boosted case.  In particular, the improvement is more considerable in the case of the hyperbolic cross basis, which is explained by the higher correlation between $\mu^2(\bm A, \bm b)$ and $\mu^2(\bm A, \tilde{\bm b})$, as reported in Table \ref{table:cavity-flow}. Additionally, for the case of hyperbolic space, the BFB results is comparable or better performance as compared to the column-pivoted QR decomposition (blue line in Figure \ref{fig:cavity_flow}). Note that the computational cost of column-pivoted QR is higher than the BFB as it requires the QR decomposition of the entire matrix $\bm A$. 

\begin{table}[ht]
	\centering
	\caption{
		Correlation coefficients between $\mu^2(\bm A, \bm b)$ and $\mu^2(\bm A, \tilde{\bm b})$ for different sampling methods under total degree or hyperbolic cross space.
		The correlation is computed based on the points shown in Figure~\ref{fig:corr_plots}.
	}
	\begin{tabular}{lccc}
	\toprule
	  Polynomial Space      & Uniform Sampling & Leverage Score Sampling & Leveraged Volume Sampling  \\
	\midrule
	Total Degree 	    & 0.66 	& 0.57 		& 0.18 \\
	Hyperbolic Cross 	& 0.99 	& 0.98 		& 0.98 \\
	\bottomrule
	\end{tabular}
	\label{table:cavity-flow}
\end{table}

\begin{figure}[H]
    \centering
    \includegraphics[width=1.0\textwidth]{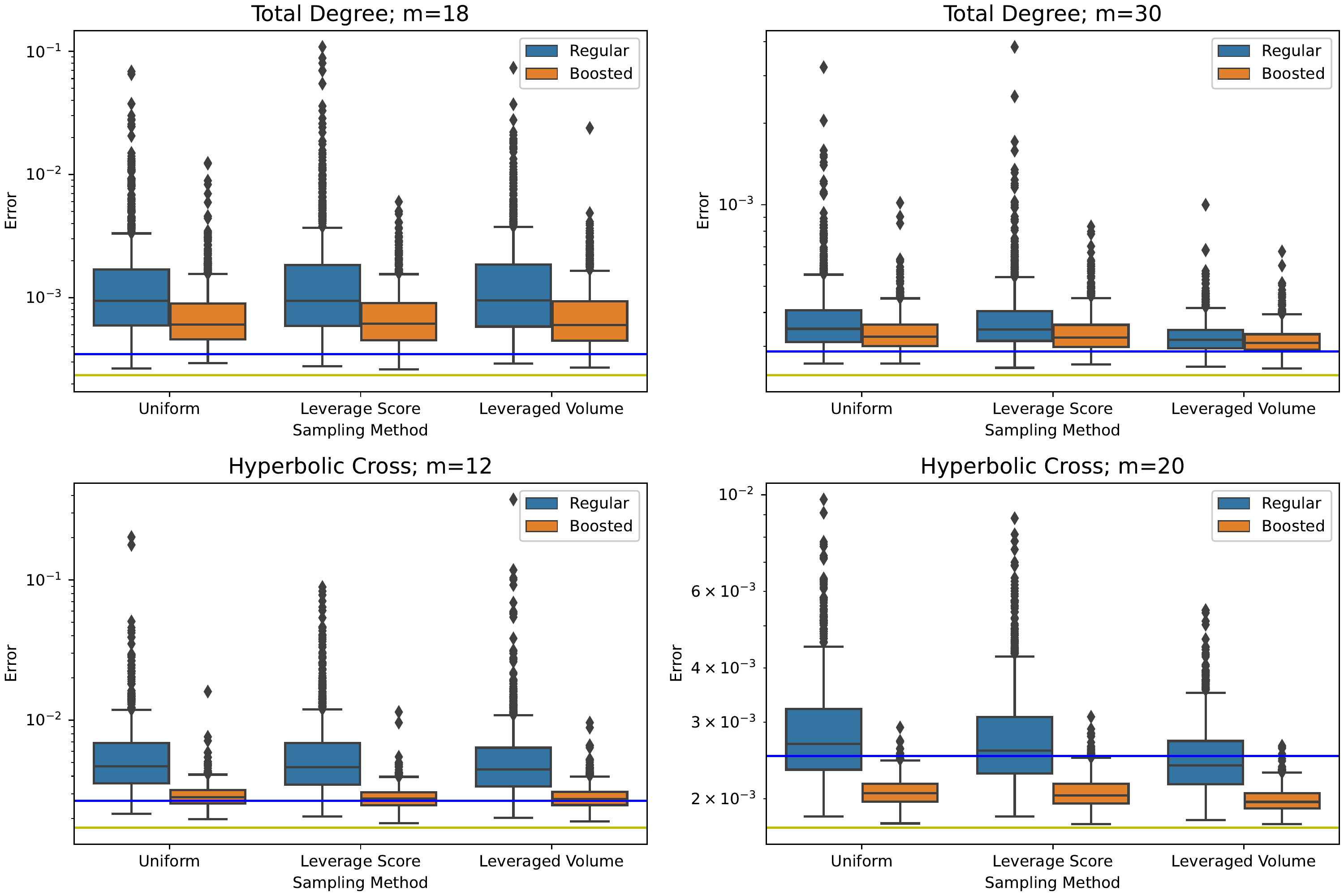}
    \caption{Relative error for different sampling methods and polynomial spaces when fitting the surrogate model to the cavity fluid flow data.
    	Yellow lines show the relative error $E$ in (\ref{eqn:rel_err}) for the unsketched solution in \eqref{myleastsquares}. 
    	Blue lines show $E$ when the coefficients $\bm x$ are computed via the QR decomposition-based method in Section~\ref{sssec:qr_sampling}.   
		The blue box plots shows the distribution of $E$ based on 1000 trials when $\bm x$ is computed as in \eqref{eq:xast-def}.
		The orange box plots shows the same things, but for the solution $\hat{\bs{x}}_{\bfqs}$ computed via Algorithm~\ref{alg:BFB}.
    }
    \label{fig:cavity_flow}
\end{figure}

\subsubsection{Composite beam} \label{sec:composite-beam}

Following \cite{hfnd18,de2020transfer,de2022neural}, we consider a plane-stress, cantilever beam with composite cross section and hollow web  as shown in Figure \ref{fig:cantilever}. 
The quantity of interest in this case is the maximum displacement of the top cord. 
The uncertain parameters of the model are $E_1, E_2, E_3, f$, where $E_1$, $E_2$ and $E_3$ are the Young's moduli of the three components of the cross section and $f$ is the intensity of the applied distributed force on the beam; see Figure~\ref{fig:cantilever}. These are assumed to be statistically independent and uniformly distributed. The dimension of the input parameter is therefore $q = 4$. Table \ref{table:beam} shows the range of the input parameters as well as the other deterministic parameters.

\begin{figure}[H]
	\centering
	\includegraphics[trim = 62mm 80mm 65mm 75mm, clip, width = 0.8\textwidth]{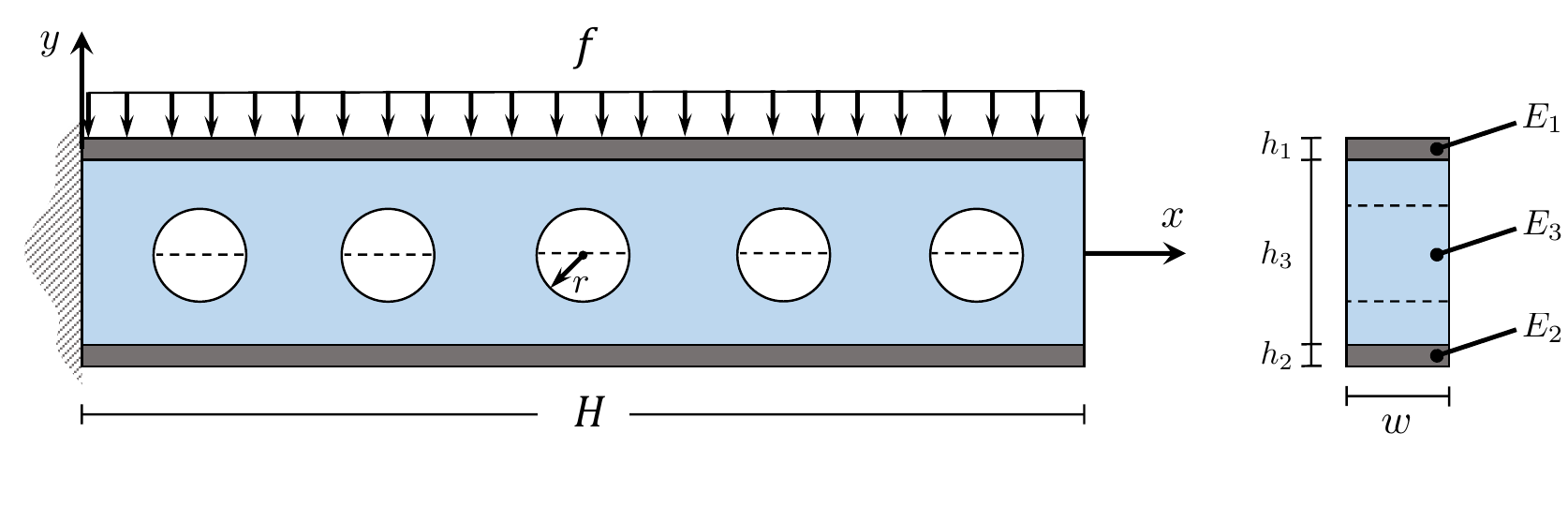}
	\caption{Cantilever beam (left) and the composite cross section (right) adapted from \cite{hampton2018PracticalError}. }
	\label{fig:cantilever}
\end{figure}

\begin{table}[htb]
	\centering
	\caption{The values of the parameters in the composite cantilever beam model. 
		The center of the holes are at $x=\{5,15,25,35,45\}$.
		The parameters $f$, $E_1$, $E_2$ and $E_3$ are drawn independently and uniformly at random from the specified intervals.
		\label{table:beam}} 
	\centering
	\begin{tabular}{ c c c c c c c c c c}   
		\toprule 
			$H$ & $h_1$ & $h_2$ & $h_3$ & $w$ & $r$ & $f$ & $E_1$ & $E_2$ & $E_3$ \\
		\midrule
			50  &  0.1 & 0.1 & 5 &  1 &  1.5 & $[9,11]$ & $[0.9\text{e6}, 1.1\text{e6}]$  &  $[0.9\text{e6}, 1.1\text{e6}]$ & $[0.9\text{e4}, 1.1\text{e4}]$  \\ 
		\bottomrule
	\end{tabular} 
	\label{table:parameters} 
\end{table}

For the cavity fluid flow problem in Section~\ref{sec:cavity-fluid-flow}, we created high- and low-fidelity solutions by changing the resolution of the grid used in the numerical solver.
For the present problem, we instead use two different models.
The high-fidelity model is based on a finite element discretization of the beam using a triangle mesh, as Figure~\ref{fig:mesh} shows. 
The low-fidelity model is derived from Euler--Bernoulli beam theory in which the vertical cross sections are assumed to remain planes throughout the deformation. 
The low-fidelity model ignores the shear deformation of the web and does not take the circular holes into account. 
Considering the Euler-Bernoulli theorem, the vertical displacement $u$ is 
\begin{equation}
	\label{eq:Euler-Bernoulli}
	EI\frac{d^4u(x)}{dx^4}=-f,
\end{equation}
where $E$ and $I$ are, respectively, the Young's modulus and the moment of inertia of an equivalent cross section consisting of a single material. We let $E=E_3$, and the width of the top and bottom sections are $w_1=(E_1/E_3)w$ and $w_2=(E_2/E_3)w$, while all other dimensions are the same, as Figure \ref{fig:cantilever} shows. The solution of \eqref{eq:Euler-Bernoulli} is 
\begin{equation}
	\label{eq:Euler_Bernoulli_soln}
	u(x)=-\frac{qH^4}{24EI}\left(\left(\frac{x}{H}\right)^4-4\left(\frac{x}{H}\right)^3+6\left(\frac{x}{H}\right)^2\right).
\end{equation}
\begin{figure}[H]
	\centering
	\includegraphics[trim = 0mm 80mm 0mm 70mm, clip, width = 0.8\textwidth]{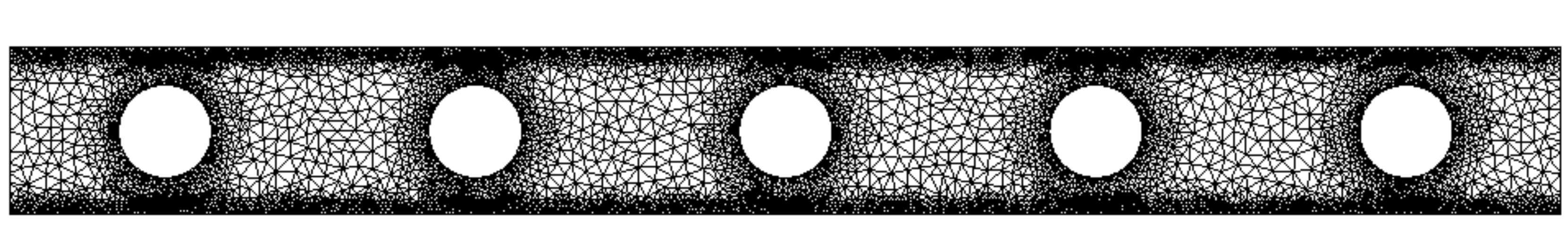}
	\caption{Finite element mesh used to generate high-fidelity solutions.  }
	\label{fig:mesh}
\end{figure}

The surrogate model is based on multivariate Legendre polynomials of maximum degree $\zeta=2$ with total degree and hyperbolic cross truncation. The corresponding spaces have $d=15$ and $d=9$ basis functions, respectively.
As in the case of the cavity flow problem, the quadrature pairs $(\bs{p}_n, w_n)$ used to construct $\bs{A}$, $\bs{b}$ and $\tilde{\bs{b}}$ are based on the nodes and weights from 10-point Gauss--Legendre rule appropriately mapped into the ranges given in Table~\ref{table:parameters}.

Figure~\ref{fig:corr_plots_beam} shows the scatter plots of $(\mu^2(\tilde{\bm b}, \bm S), \mu^2(\bm b, \bm S))$ when repeating the experiment in Section~\ref{sec:synthetic-experiments} for the two different polynomial spaces and three different random sampling approaches.
Each plot is based on 100 sketches with $m=2d$, i.e., $m=30$ and $m=18$ samples used for the total degree and hyperbolic cross spaces, respectively. Table \ref{table:beam-corr} reports the correlation coefficient between $\mu^2(\bm b,\bm S)$ and $\mu^2(\tilde{\bm b},\bm S)$, indicating an overall high correlation in all cases.

\begin{figure}[H]
	\centering
	\includegraphics[width=1.0\textwidth]{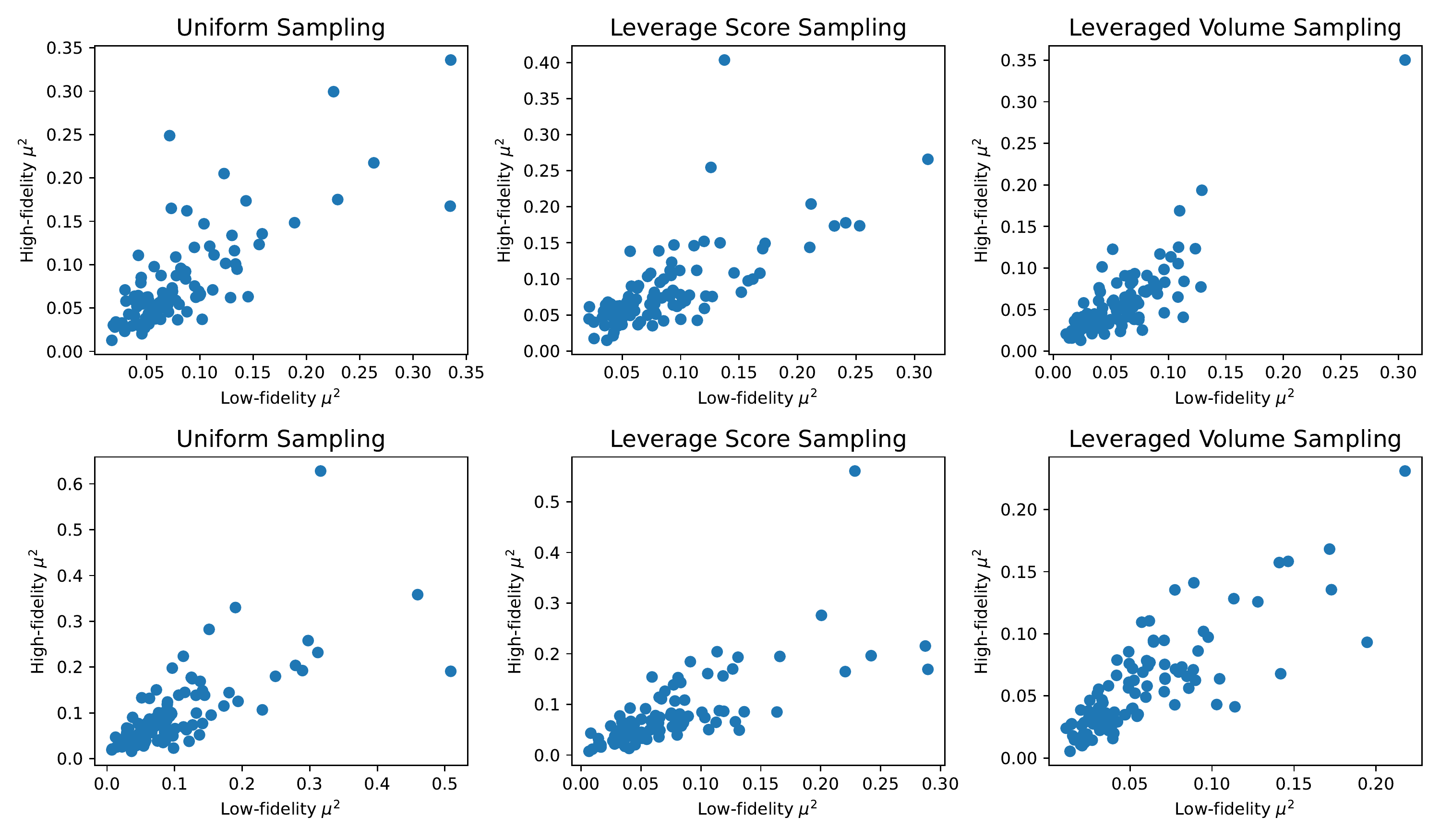}
	\caption{
		Scatter plots of the square of the optimality coefficient for high- and low-fidelity data from the composite beam problem for different polynomial spaces (top: total degree; bottom: hyperbolic cross) and types of sampling.
		Each point is equal to $(\mu^2(\tilde{\bm b}, \bm S), \mu^2(\bm b, \bm S))$ for one realization of the sketch $\bs{S}$, and each subplot contains 100 points (i.e., is based on 100 sketch realizations).
		For the total degree space $m=30$ samples are used and for the hyperbolic cross space $m=18$ samples are used.
		The corresponding correlation coefficients are presented in Table \ref{table:beam-corr}.
	}
	\label{fig:corr_plots_beam}
\end{figure}

\begin{table}[ht]
	\centering
	\caption{
		Correlation coefficient between $\mu^2(\bm A, \bm b)$ and $\mu^2(\bm A, \tilde{\bm b})$ for different sampling methods under total degree or hyperbolic cross space.
		The correlation is computed based on the points shown in Figure~\ref{fig:corr_plots_beam}.
	}
	\begin{tabular}{lccc}
	\toprule
	  Polynomial Space  & Uniform Sampling & Leverage Score Sampling & Leveraged Volume Sampling  \\
	\midrule
	Total Degree 	    & 0.77 	& 0.69 		& 0.84 \\
	Hyperbolic Cross 	& 0.72 	& 0.73 		& 0.82 \\
	\bottomrule
	\end{tabular}
	\label{table:beam-corr}
\end{table}

Next, we run Algorithm~\ref{alg:BFB} with $L=10$ sketches and $m$ chosen to be $m=1.2d$ and $m=2d$.
Figure~\ref{fig:beam} shows the results from running the algorithm 1000 times for each of the different choices of polynomial space, number of samples $m$, and random sampling approach. We observe that the BFB performance is superior to that of the non-boosted implementation as it leads to smaller variance of the error and fewer outliers with smaller deviation from the mean performance. In this example, the BFB leads to comparable accuracy as the column-pivoted QR sketch, but with smaller sketching cost. As in the case of the cavity flow, the results corroborate the discussion below Theorem~\ref{thm:Gaussian}, in that the BFB improves the regression accuracy when $\cor(\mu^2(\bm b,\bm S),\mu(\tilde{\bm b},\bm S))$ is large. 

\begin{figure}[H]
	\centering
	\includegraphics[width=1.0\textwidth]{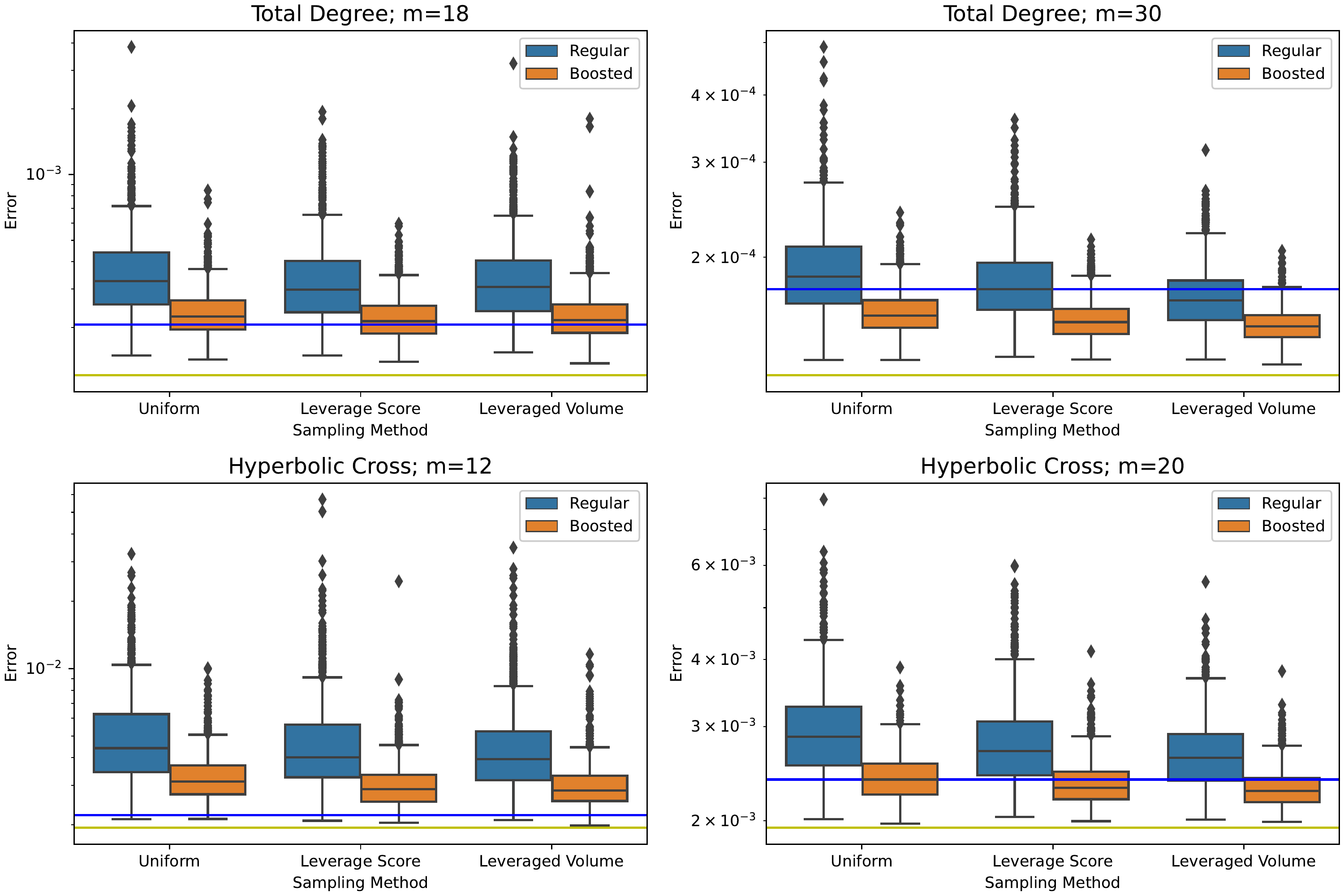}
	\caption{Relative error for different sampling methods and polynomial spaces when fitting the surrogate model to the beam problem data.
    	Yellow lines show the relative error $E$ in (\ref{eqn:rel_err}) for the unsketched solution in \eqref{myleastsquares}. 
    	Blue lines show $E$ when the coefficients $\bm x$ are computed via the QR decomposition-based method in Section~\ref{sssec:qr_sampling}.   
		The blue box plots shows the distribution of $E$ based on 1000 trials when $\bm x$ is computed as in \eqref{eq:xast-def}.
		The orange box plots shows the same things, but for the solution $\hat{\bs{x}}_{\bfqs}$ computed via Algorithm~\ref{alg:BFB}.}
	\label{fig:beam}
\end{figure}

\section{Conclusion}
\label{sec:conclusion}

This work was concerned with the construction of (polynomial) emulators of parameter-to-solution maps of PDE problems via sketched least-squares regression. Sketching is a design of experiments approach that aims to improve the cost of building a least squares solution in terms of reducing the number of samples needed --- when the cost of generating data is high --- or the cost of generating a least squares solution –-- when data size is substantial. Focusing on the former case, we have proposed a new boosting algorithm to compute a sketched least squares solution. 

The procedure consisted in identifying the best sketch from a set of candidates used to construct least squares regression of the low-fidelity data and applying this {\it optimal} sketch to the regression of high-fidelity data. The bi-fidelity boosting (BFB) approach limits the required sample complexity to $\sim d \log d$ high-fidelity data, where $d$ is the size of the (polynomial) basis. We have provided theoretical analysis of the BFB approach identifying assumptions on the low- and high-fidelity data under which the BFB leads to improvement of the solution relative to non-boosted regression of the high-fidelity data. We have also provided quantitative bounds on the residual of the BFB solution relative to the full, computationally expensive solution. We have investigated the performance of BFB on manufactured and PDE data from fluid and solid mechanics.  These cover sketching strategies based on leverage score and leveraged volume sampling, for truncated Legendre polynomials of both total degree and hyperbolic cross type. All tests illustrated the efficacy of BFB in reducing the residual --- as compared to the non-boosted implementation –-- and validate the theoretical results.

The present study was focused on the case of (weighted) least squares polynomial regression. When the regression coefficients are sparse, methods based on compressive sampling have proven efficient in reducing the sample complexity below the size of the polynomial basis; see, e.g., \cite{doostan2011non,adcock2022sparse}. As interesting future research direction is to extend the BFB strategy to such under-determined cases, for instance, using the approach of \cite{diaz2018sparse}.

\section*{Acknowledgments} 

This work was supported by the AFOSR awards FA9550-20-1-0138 and FA9550-20-1-0188 with Dr. Fariba Fahroo as the program manager. The views expressed in the article do not necessarily represent the views of the AFOSR or the U.S. Government.


\newpage
\appendix

\section{Efficient leverage score sampling of certain design matrices} \label{sec:lev-score-sampling-alg}

In this section, we describe the key elements of the sampling approach developed in \citep{malik2022FastAlgorithms} as it applies to the problems we consider in this paper.
The discussion here will consider the design matrices discussed in Section~\ref{sec:experiments-pdf-datasets}.
Using the same notation as in that section, define the matrices $\bs{A}_k$ for $k \in [q]$ elementwise via
\begin{equation}
    \bm A_k(n_k, j_k) = \sqrt{w_{k, n_k}} \psi_{j_k}(p_{k, n_k}), \;\;\;\; n_k \in [N_k], \; j_k \in [\zeta].
\end{equation}
Next, define
\begin{equation}
    \bs{A}_\text{TP} \defeq \bs{A}_1 \otimes \cdots \otimes \bs{A}_q,
\end{equation}
where $\otimes$ denotes the Kronecker product; see Section~12.3 of \citep{golub2013MatrixComputations} for a definition.
The design matrices corresponding to total degree and hyperbolic cross polynomial spaces discussed in Section~\ref{sec:experiments-pdf-datasets} are made up of a subset of the columns of $\bs{A}_\text{TP}$.
In particular, using Matlab indexing notation, they can be written as
\begin{equation}
    \bs{A} = \bs{A}_\text{TP}(:, \bs{v}),
\end{equation}
where $\bs{v}$ is a vector containing distinct column indices of $\bs{A}_\text{TP}$.
The sampling scheme we discuss requires the additional assumption that the entries in $\bs{v}$ are arranged in increasing order.
The columns of $\bs{A}$ can always be permuted to ensure that this is possible when $\bs{A}$ corresponds to a total degree or hyperbolic cross space.
Such a permutation will not change the least squares problem since  
it will only permute the order of the entries in the solution vector, and is therefore something that can always be done.

Note that a column index $c$ of $\bs{A}_\text{TP}$ corresponds to a multi-index $(c_1, \ldots, c_q)$ such that
\begin{equation}
    \bs{A}_{\text{TP}}(:, c) = \bs{A}_1(:, c_1) \otimes \cdots \otimes \bs{A}_{q}(:, c_q).
\end{equation}
Each row index $r$ of $\bs{A}_\text{TP}$ corresponds to a multi-index $(r_1, \ldots, r_q)$ in a similar fashion.

Algorithm~\ref{alg:lev-score-samp} outlines the sampling algorithm.
We provide some intuition for why the algorithm works and refer the reader to \citep{malik2022FastAlgorithms} for a rigorous treatment.
Note that $\bs{A}$ is full rank and therefore $\rank(\bs{A}) = d$. 
Let $\bs{Q} \bs{R} = \bs{A}$ be a compact QR decomposition (i.e., such that $\bs{Q}$ has $d$ columns and $\bs{R}$ has $d$ rows).
Recall that the leverage score sampling distribution satisfies
\begin{equation} \label{eq:lev-score-dist}
    \bs{p}(i) = \frac{\| \bs{Q}(i, :) \|_2^2}{d}.
\end{equation}
Instead of drawing a sample according to the distribution above, we may instead draw a single column $\bs{Q}(:,j)$ of $\bs{Q}$ uniformly at random and instead draw a sample according to the probability distribution defined by $\tilde{\bs{p}}_j(i) = (\bs{Q}(i,j))^2$.
To see this, let $\tilde{I}$ be a random row index drawn according to this alternate strategy.
Moreover, let $J \sim \Uniform([d])$ be the random column index, and let $\tilde{I}_j$ be a random row index drawn according to $\tilde{\bs{p}}_j$.
Then we have
\begin{equation}
    \Pb(\tilde{I} = i) 
    = \sum_{j = 1}^d \Pb(\tilde{I} = i \mid J = j) \, \Pb(J=j) 
    = \sum_{j = 1}^d \Pb(\tilde{I}_j = i) \, \Pb(J=j) 
    = \sum_{j = 1}^d (\bs{Q}(i,j))^2 \, \frac{1}{d} 
    = \frac{\|\bs{Q}(i, :)\|_2^2}{d}
    = \bs{p}(i).
\end{equation}
This shows that the alternate sampling strategy indeed draws samples according to the leverage score sampling distribution.
This is the sampling strategy that our algorithm uses.
Moreover, it uses two additional fact:
\begin{enumerate}[(i)]
    \item When $\bs{A}$ has the particular structure assumed in this section, then the $c$th column of $\bs{Q}$ satisfies
    \begin{equation} \label{eq:kronecker-trick}
        \bs{Q}(:,c) = \bs{Q}_1(:, c_1) \otimes \cdots \otimes \bs{Q}_q(:, c_q),
    \end{equation}
    where $\bs{Q}_1, \ldots, \bs{Q}_{q}$ are defined in line~\ref{line:small-qr} in Algorithm~\ref{alg:lev-score-samp}.

    \item Due to \eqref{eq:kronecker-trick}, drawing a row index $r$ according to $\tilde{\bs{p}}_j$ is equivalent to drawing a multi-index $(r_1,\ldots,r_q)$ according to a product distribution with each $r_k$ drawn independently according to the distribution $((\bs{Q}_k(r_k, j_k))^2)_{r_k}$ where $(j_1, \ldots, j_q)$ is the column multi-index corresponding to $j$.
\end{enumerate}
Fact (i) makes it possible to sample according to the alternate sampling strategy without every needing to compute the QR decomposition of the large matrix $\bs{A}$.
A more general version of this fact appears in Proposition~4.4 of \citep{malik2022FastAlgorithms}.
Fact (ii) further makes it possible to sample according to $\tilde{\bs{p}}_j$ without needing to form that probability vector which is of length $\prod_k N_k$.

\begin{algorithm}[ht]
		\DontPrintSemicolon
		\caption{Efficient leverage score sampling of total degree and hyperbolic cross design matrices\label{alg:lev-score-samp}}
		\KwIn{Matrices $\bs{A}_1, \ldots, \bs{A}_q$, index vector $\bs{v}$, number of samples $m$}
		\KwOut{Vector $\bs{s} \in [\prod_k N_k]^m$ of $m$ samples drawn from row indices of $\bs{A}$}
		\begin{algorithmic}[1]
			\FOR{$k \in [q]$}{
			    \STATE Compute compact QR decomposition $\bs{Q}_k \bs{R}_k = \bs{A}_k$ \label{line:small-qr}
			}
			\ENDFOR
			\FOR{$i \in [m]$}{
			    \STATE Draw an entry $j$ from $\bs{v}$ uniformly at random
			    \STATE Compute the multi-index $(j_1, \ldots, j_q)$ corresponding to $j$
			    \FOR{$k \in [q]$}{
			        \STATE Construct the probability distribution $\bs{p} = (( \bs{Q}_{k}(r_k, j_k) )^2)_{r_k} \in \Rb^{N_k}$
			        \STATE Draw an index $r_k \in [N_k]$ according to the distribution $\bs{p}$
			    }
			    \ENDFOR
			    \STATE Set the $i$th sample $\bs{s}(i)$ equal the row index corresponding to the row multi-index $(r_1,\ldots,r_q)$
			}
			\ENDFOR
		\RETURN Vector of samples $\bs{s}$
		\end{algorithmic}
\end{algorithm}

\section{Proof of Theorem \ref{thm:Gaussian}}\label{app:a}

The proof of Theorem \ref{thm:Gaussian} relies on the following lemmas:

\begin{lemma}\label{lem:cor}
Let $X$ and $Y$ be two (nonconstant) random variables defined on the same probability space. 
The correlation coefficient between $X$ and $Y$, $\cor(X,Y)$, is bounded from below as
\begin{equation}
	\cor(X, Y) \geq \sqrt{\frac{\V[X]}{\V[Y]}} - \sqrt{\frac{\V[Y-X]}{\V[Y]}}.\label{corcorb}
\end{equation}
\end{lemma}

\begin{proof}
It follows from direct computation that 
\begin{equation}
\begin{aligned}
\cor(X, Y) 
&= \frac{\E[XY]-\E[X]\E[Y]}{\sqrt{\V[X]\V[Y]}} \\
& = \frac{\E[X^2]-\E[X]^2}{\sqrt{\V[X]\V[Y]}}+\frac{\E[X(Y-X)]-\E[X]\E[Y-X]}{\sqrt{\V[X]\V[Y]}}\\
& = \sqrt{\frac{\V[X]}{\V[Y]}} + \cor(X, Y-X)\sqrt{\frac{\V[Y-X]}{\V[Y]}}\\
&\geq \sqrt{\frac{\V[X]}{\V[Y]}} - \sqrt{\frac{\V[Y-X]}{\V[Y]}},
\end{aligned}
\end{equation}
where the last inequality uses $\cor(X, Y-X)\geq -1$. 
\end{proof}

\begin{lemma}\label{lem:gaussian}
Let $\bm\xi\sim\mathcal N(\bm 0, \bm I_n)$ be a standard Gaussian vector in $\R^n$. 
For any $\bm w, \bm z\in\R^n$, 
\begin{equation}
	\E[\langle\bm w, \bm\xi\rangle^2\langle\bm z, \bm\xi\rangle^2] = 2\langle\bm w, \bm z\rangle^2+\|\bm w\|_2^2\|\bm z\|_2^2.\label{mycal}
\end{equation}
\end{lemma}

\begin{proof}
The proof follows from a direct application of Wick's formula \cite{wick1950evaluation}. 
Denote $X_1 = \langle\bm w, \bm\xi\rangle$ and $X_2 = \langle\bm z, \bm\xi\rangle$. 
It is easy to verify that 
\begin{equation}
\begin{pmatrix}
X_1\\
X_2
\end{pmatrix}\sim\mathcal N(\bm 0, \bm K),
\;\;\;\; \bm K = \begin{pmatrix}
\|\bm w\|_2^2 & \langle \bm w, \bm z\rangle\\
 \langle \bm w, \bm z\rangle & \|\bm z\|_2^2
\end{pmatrix}.
\end{equation}
By Wick's formula, 
\begin{equation}
	\E[\langle\bm w, \bm\xi\rangle^2\langle\bm z, \bm\xi\rangle^2] = \E[X_1^2X_2^2] =  2\E[X_1X_2]^2 + \E[X_1^2]\E[X_2^2] = 2\langle\bm w, \bm z\rangle^2+\|\bm w\|_2^2\|\bm z\|_2^2.
\end{equation} 
\end{proof}

\begin{proof}[Proof of Theorem \ref{thm:Gaussian}]
Since correlation coefficients are scale-invariant, and both $\bm b$ and $\tilde{\bm b}$ are fixed, 
\begin{equation}
	\cor(\mu^2( \bm b, \bm S),\mu^2(\tilde{\bm b}, \bm S)) = \cor\left(\frac{r_{\bm S}^2(\bm A, \bm b)-r^2(\bm A, \bm b)}{\|\bm b\|_2^2}, \frac{r_{\bm S}^2(\bm A, \tilde{\bm b})-r^2(\bm A, \tilde{\bm b})}{\|\tilde{\bm b}\|_2^2}\right).
\end{equation}
Without loss of generality, we assume $\|\bm b\|_2 = \|\tilde{\bm b}\|_2 = 1$, so that $\bm b_\PP = \bm b$, $\tilde{\bm b}_\PP = \tilde{\bm b}$.  

Let 
\begin{equation}
\begin{aligned}
	&X = r_{\bm S}^2(\bm A, \bm b)-r^2(\bm A, \bm b) = \|(\bm S\bm Q)^\dagger\bm S\bm Q_\perp\bm Q^T_\perp\bm b\|_2^2\\
	& Y = r_{\bm S}^2(\bm A, \tilde{\bm b})-r^2(\bm A, \tilde{\bm b}) = \|(\bm S\bm Q)^\dagger\bm S\bm Q_\perp\bm Q^T_\perp\tilde{\bm b}\|_2^2.
\end{aligned}
\end{equation}
To apply Lemma \ref{lem:cor}, it suffices to estimate $\V[X]/\V[Y]$ and $\V[Y-X]/\V[Y]$.

First of all, due to the rotation invariance of joint Gaussians, 
\begin{equation}
\begin{aligned}
	\bm G_1 &\defeq \sqrt{m}\bm S\bm Q\in\R^{m\times d}, \\
	\bm G_2 &\defeq \sqrt{m}\bm S\bm Q_\perp\in\R^{m\times (N-d)}
\end{aligned}
\end{equation}
are independent Gaussian random matrices, i.e., $(\bm S\bm Q)^\dagger\bm S\bm Q_\perp\bm Q^T_\perp = \bm G_1^\dagger\bm G_2\bm Q^T_\perp$, and 
\begin{equation}
\begin{aligned}
	\E[X] 
	&= \E\left[\tr(\bm G_1^\dagger\bm G_2\bm Q^T_\perp\bm b\bm b^T\bm Q_\perp\bm G_2^T{\bm G_1^\dagger}^T)\right] \\
	&= \E\left[\tr(\bm G_1^\dagger\E[\bm G_2\bm Q^T_\perp\bm b\bm b^T\bm Q_\perp\bm G_2^T]{\bm G_1^\dagger}^T)\right] \\
	&= \|\bm Q^T_\perp\bm b\|_2^2\E\left[\tr(\bm G_1^\dagger{\bm G_1^\dagger}^T)\right]\\
	&=\|\bm Q^T_\perp\bm b\|_2^2\E\left[\tr((\bm G_1^T\bm G_1)^{-1})\right],
\end{aligned}
\end{equation}
where we have used that $\E[\bm G_2\bm Q^T_\perp\bm b\bm b^T\bm Q_\perp\bm G_2^T] = \|\bm Q^T_\perp\bm b\|_2^2\bm I_m$. 

Note $\bm G_1^T\bm G_1$ is a Wishart matrix with dimension $d$ and degrees of freedom $m$, i.e. $\bm W=\bm G_1^T\bm G_1\sim W_d(\bm I_d, m)$.
Consequently, $\E[\bm W^{-1}] = \frac{1}{m-d-1}\bm I_d$ if $m>d+1$, and 
\begin{equation} \label{eX}
	\E[X] = \|\bm Q^T_\perp\bm b\|_2^2\frac{d}{m-d-1}.
\end{equation}
Similarly,
\begin{equation} \label{eY}
	\E[Y] = \|\bm Q^T_\perp\tilde{\bm b}\|_2^2\frac{d}{m-d-1}.
\end{equation}
Note $\bm G_1^\dagger\bm G_2\bm Q^T_\perp\bm a \stackrel{\mathcal D}{=} \|\bm Q^T_\perp\bm a\|_2(\bm G_1^T\bm G_1)^{-1}\bm G_1^T\bm\xi$ for every $\bm a\in\R^N$, where $\bm\xi\sim\mathcal N(\bm 0, \bm I_m)$ is independent of $\bm G_1$. 
If we denote $\bm G = (\bm G_1^T\bm G_1)^{-1}\bm G_1^T$, with rows denoted by $\bm g_i, i\in [d]$, then
\begin{equation} \label{X4}
\begin{aligned}
	\E[X^2] 
	&= \E[\|\|\bm Q^T_\perp\bm b\|_2\bm G\bm\xi\|_2^4] \\
	&= \|\bm Q^T_\perp\bm b\|_2^4\E\left[\left(\sum_{i=1}^d\langle\bm g_i, \bm\xi\rangle^2\right)^2\right]\\
	&=  \|\bm Q^T_\perp\bm b\|_2^4\left(\sum_{i=1}^d\E[\langle\bm g_i, \bm\xi\rangle^4] + \sum_{i\neq j}\E[\langle\bm g_i, \bm\xi\rangle^2\langle\bm g_j, \bm\xi\rangle^2]\right)\\
	& \stackrel{\eqref{mycal}}{=} \|\bm Q^T_\perp\bm b\|_2^4\left(3\sum_{i=1}^d\E[\|\bm g_i\|_2^4] + \sum_{i\neq j} (2\E[\langle\bm g_i, \bm g_j\rangle^2]+\E[\|\bm g_i\|_2^2\|\bm g_j\|_2^2])\right)\\
	& = \|\bm Q^T_\perp\bm b\|_2^4\left(2\E[\|\bm G\bm G^T\|_F^2] + \E[\tr(\bm G\bm G^T)^2]\right)\\
	& = \|\bm Q^T_\perp\bm b\|_2^4\left(2\E[\|(\bm G_1^T\bm G_1)^{-1}\|_F^2] + \E[\tr((\bm G_1^T\bm G_1)^{-1})^2]\right). 
\end{aligned}
\end{equation}
To explicitly compute \eqref{X4}, we use the following moments formulas of inverse Wishart distributions \cite[Theorem 2.4.14]{kollo2005advanced}:
\begin{equation}
\begin{aligned}
	\E[\bm W^{-1}\bm W^{-1}] &= \left(\frac{d}{(m-d)(m-d-3)}+\frac{d}{(m-d)(m-d-1)(m-d-3)}\right)\bm I_d\\
	\text{Cov}(\bm W^{-1}_{ii}, \bm W^{-1}_{jj}) &= \frac{2+2(m-d-1)\delta_{ij}}{(m-d)(m-d-1)^2(m-d-3)}.
\end{aligned}
\end{equation}
Therefore,
\begin{equation}
	\E[\|(\bm G_1^T\bm G_1)^{-1}\|_F^2] 
	= \tr(\E[\bm W^{-1}\bm W^{-1}]) = \frac{d^2}{(m-d-1)(m-d-3)}\simeq\frac{d^2}{(m-d-1)^2}\\
\end{equation}
and
\begin{equation}
\begin{aligned}
	\E[\tr((\bm G_1^T\bm G_1)^{-1})^2] 
	&= \sum_{i,j\in [d]}\E[\bm W^{-1}_{ii}\bm W^{-1}_{jj}]\\
	& = \sum_{i,j\in [d]}(\text{Cov}(\bm W^{-1}_{ii}, \bm W^{-1}_{jj}) +\E[\bm W^{-1}_{ii}]\E[\bm W^{-1}_{jj}])\\
	& = \frac{d^2}{(m-d-1)^2} + \frac{2d}{(m-d-1)^2(m-d-3)} + \frac{2(d^2-d)}{(m-d)(m-d-1)^2(m-d-3)}\\
	&\simeq \frac{d^2}{(m-d-1)^2},
\end{aligned}
\end{equation}
where $a_m\simeq b_m$ if $\lim_{m\to\infty}a_m/b_m = 1$.
Substituting these back into \eqref{X4} yields
\begin{equation}
	\E[X^2]\simeq \|\bm Q^T_\perp\bm b\|_2^4\frac{3d^2}{(m-d-1)^2}.\label{X}
\end{equation}
Replacing $\bm b$ by $\tilde{\bm b}$ in the above computation gives a similar estimate for $\E[Y^2]$: 
\begin{equation}
	\E[Y^2]\simeq \|\bm Q^T_\perp\tilde{\bm b}\|_2^4\frac{3d^2}{(m-d-1)^2}.\label{Y}
\end{equation}
Combining \eqref{X}, \eqref{Y} with \eqref{eX} and \eqref{eY} produces
\begin{equation}
\begin{aligned}
	&\V[X]\simeq\|\bm Q^T_\perp\bm b\|_2^4\frac{2d^2}{(m-d-1)^2}, \\
	&\V[Y]\simeq \|\bm Q^T_\perp\tilde{\bm b}\|_2^4\frac{2d^2}{(m-d-1)^2},
\end{aligned}
\end{equation}
which implies
\begin{equation}	
	\frac{\V[X]}{\V[Y]}\simeq \frac{\|\bm Q^T_\perp\bm b\|_2^4}{\|\bm Q^T_\perp\tilde{\bm b}\|_2^4}.\label{ratio}
\end{equation}
On the other hand, using Cauchy--Schwarz inequality, 
Moreover, we have
\begin{equation} \label{X-Y}
\begin{aligned}
	\V[Y-X] 
	&\leq \E[(Y-X)^2] \\
	&= \E\big[ (X^{\frac{1}{2}} + Y^{\frac{1}{2}})^2 \cdot (X^{\frac{1}{2}} - Y^{\frac{1}{2}})^2 \big] \\
	&= \E\big[ (X^{\frac{1}{2}} + Y^{\frac{1}{2}})^2  \cdot (\|\bm G_1^\dagger\bm G_2\bm Q^T_\perp\bm b\|_2 - \|\bm G_1^\dagger\bm G_2\bm Q^T_\perp\tilde{\bm b}\|_2)^2 \big] \\
	&\leq  \E[(X^{\frac{1}{2}} + Y^{\frac{1}{2}})^2 \cdot \|\bm G_1^\dagger\bm G_2\bm Q^T_\perp(\bm b\pm\tilde{\bm b})\|_2^2],
\end{aligned}
\end{equation}
where the last inequality follows from the reverse triangle inequality.
Furthermore, using the inequality of arithmetic and geometric means followed by the Cauchy--Schwarz inequality, we have
\begin{equation} \label{X-Y-2}
\begin{aligned}
	&\E[(X^{\frac{1}{2}} + Y^{\frac{1}{2}})^2 \cdot \|\bm G_1^\dagger\bm G_2\bm Q^T_\perp(\bm b\pm\tilde{\bm b})\|_2^2 \\
	&\leq  2\E[(X+Y) \cdot \|\bm G_1^\dagger\bm G_2\bm Q^T_\perp(\bm b\pm\tilde{\bm b})\|_2^2]\\
	&= 2 \E[X \cdot \|\bm G_1^\dagger\bm G_2\bm Q^T_\perp(\bm b\pm\tilde{\bm b})\|_2^2]  + 2 \E[Y \cdot \|\bm G_1^\dagger\bm G_2\bm Q^T_\perp(\bm b\pm\tilde{\bm b})\|_2^2] \\
	&\leq 2\sqrt{\E[X^2] \cdot \E[\|\bm G_1^\dagger\bm G_2\bm Q^T_\perp(\bm b\pm\tilde{\bm b})\|_2^4]} + 2\sqrt{\E[Y^2] \cdot \E[\|\bm G_1^\dagger\bm G_2\bm Q^T_\perp(\bm b\pm\tilde{\bm b})\|_2^4]}.
\end{aligned}
\end{equation}
Combining \eqref{X-Y} and \eqref{X-Y-2} yields
\begin{equation} \label{eq:V-bound}
	\V[Y-X] 
	\leq 2\sqrt{\E[X^2] \cdot \E[\|\bm G_1^\dagger\bm G_2\bm Q^T_\perp(\bm b\pm\tilde{\bm b})\|_2^4]} 
	+ 2\sqrt{\E[Y^2] \cdot \E[\|\bm G_1^\dagger\bm G_2\bm Q^T_\perp(\bm b\pm\tilde{\bm b})\|_2^4]}.
\end{equation}
A similar argument as \eqref{X} shows that
\begin{equation}
	\E[\|\bm G_1^\dagger\bm G_2\bm Q^T_\perp(\bm b\pm\tilde{\bm b})\|_2^4]\simeq \|\bm Q^T_\perp(\bm b\pm\tilde{\bm b})\|_2^4\frac{3d^2}{(m-d-1)^2}.\label{Sim}
\end{equation}
Plugging \eqref{Sim} into \eqref{eq:V-bound} together with the previous estimates yields that, asymptotically, 
\begin{equation}
	\frac{\V[Y-X]}{\V[Y]} \leq \frac{2\left(\|\bm Q^T_\perp\bm b\|_2^2+\|\bm Q^T_\perp\tilde{\bm b}\|_2^2\right)\|\bm Q^T_\perp(\bm b\pm\tilde{\bm b})\|_2^2}{\|\bm Q^T_\perp\tilde{\bm b}\|_2^4}\cdot\frac{3d^2}{2d^2}\leq \frac{6\|\bm Q^T_\perp(\bm b\pm\tilde{\bm b})\|_2^2}{\|\bm Q^T_\perp\tilde{\bm b}\|_2^4}\label{ratio2},
\end{equation}
where the last inequality follows from $\|\bm b\|_2 = \|\tilde{\bm b}\|_2 = 1$.
Appealing to Lemma \ref{lem:cor}, 
\begin{equation}
	\liminf_{m\to\infty}\cor(X, Y)\geq \frac{\|\bm Q^T_\perp\bm b\|_2^2-\sqrt{6}\min\{\|\bm Q^T_\perp(\bm b\pm\tilde{\bm b})\|_2\}}{\|\bm Q^T_\perp\tilde{\bm b}\|_2^2}.
\end{equation}
\eqref{mybound} follows by noting $\|\bm Q_\perp^T\bm a\|_2 = \|\bm P_{\bm Q_\perp}\bm a\|_2$ for $\bm a\in\R^N$.

To prove \eqref{mybound1}, we use Proposition \ref{prop:cor} (i.e. \eqref{bbc}) to lower bound $\|\bm Q^T_\perp\tilde{\bm b}\|_2^2$:
\begin{equation}
	\varphi \leq \|\bm Q^T_\perp\tilde{\bm b}\|_2 +  \kappa\Longrightarrow (\varphi - \kappa)^2\leq\|\bm Q^T_\perp\tilde{\bm b}\|_2^2\leq\|\tilde{\bm b}\|^2_2=1.
\end{equation}
Also, $\|\bm Q^T_\perp\bm b\|_2^2= 1-\kappa^2$ and
\begin{equation}
	\min\{\|\bm Q^T_\perp(\bm b\pm\tilde{\bm b})\|_2\}\leq\min\{\|\bm b\pm\tilde{\bm b}\|_2\} = \sqrt{2-2\varphi}.
\end{equation}
Hence, 
\begin{equation}
	\liminf_{m\to\infty}\cor(X, Y)\geq (1-\kappa^2)-\frac{\sqrt{12(1-\varphi)}}{(\varphi-\kappa)^2},
\end{equation}
completing the proof. 
\end{proof}

\section{Proof of Theorem \ref{thm:sketches}}\label{app:a0}

We first prove the case of the sub-Gaussian sketches.  
According to Lemma \ref{lem:Drineas2011}, it suffices to verify the conditions \eqref{mycond}.

Note $\sqrt{m}\bm S\bm Q\in\R^{m\times d}$ is a random matrix whose rows are i.i.d.\ isotropic random vectors in $\R^d$, with the sub-Gaussian norm $\lesssim K$ (this follows from Definition~3.4.1 and Proposition~2.6.1 in \cite{vershynin2018HighdimensionalProbability}).
Applying \citep[Theorem 4.6.1]{vershynin2018HighdimensionalProbability} to the matrix $\sqrt{m} \bs{S} \bs{Q}$ and using the fact that $\sigma_{\min}(\sqrt{m} \bs{S} \bs{Q}) = \sqrt{m} \sigma_{\min}(\bs{S} \bs{Q})$, we find that if $m\gtrsim K^4d\log\left(4L/\delta\right)$, then with probability at least $1-\delta/(2L)$, the first condition in \eqref{mycond} is satisfied.

For the second condition in \eqref{mycond}, we write the $i$-th component of $\bm Q^T\bm S^T\bm S\bm h$ as
\begin{equation} \label{eq:sum-of-subexponentials}
	\bm q_i^T\bm S^T\bm S\bm h 
	= \frac{1}{m}\sum_{j\in [m]} \langle\sqrt{m}\bm s_j, \bm q_i\rangle \langle\sqrt{m}\bm s_j, \bm h\rangle, \;\;\;\; i\in [d].
\end{equation}
Both $\langle\sqrt{m}\bm s_j, \bm q_i\rangle$ and $\langle\sqrt{m}\bm s_j, \bm h\rangle$ are sub-Gaussian random variables \citep[Proposition~2.6.1]{vershynin2018HighdimensionalProbability}.
Therefore, 
\begin{equation}
		\|\langle\sqrt{m}\bm s_j, \bm q_i\rangle\langle\sqrt{m}\bm s_j, \bm h\rangle\|_{\psi_1} 
		\leq \|\langle\sqrt{m}\bm s_j, \bm q_i\rangle\|_{\psi_2}\|\langle\sqrt{m}\bm s_j, \bm h\rangle\|_{\psi_2}
		\leq \|\sqrt{m}\bm s_j\|_{\psi_2}^2
		\lesssim K^2, 
\end{equation}
where the first inequality follows from \citep[Lemma~2.7.7]{vershynin2018HighdimensionalProbability}, 
the second inequality follows from \citep[Definition~3.4.1]{vershynin2018HighdimensionalProbability},
and the final inequality follows from an application of \citep[Proposition~2.6.1]{vershynin2018HighdimensionalProbability}.
Moreover, since $\bs{h} \perp \range(\bs{Q})$ it is easy to verify that the summands in \eqref{eq:sum-of-subexponentials} are all zero-mean.
By Bernstein's inequality \cite[Corollary 2.8.3]{vershynin2018HighdimensionalProbability}, if $m\gtrsim K^4d\log (4dL/\delta)/\e$, with probability at least $1-\delta/(2dL)$, $|\bm q_i^T\bm S^T\bm S\bm h|\leq\sqrt{\e/(2d)}$.
Taking a union bound over $i \in [d]$ yields that, with probability at least $1-\delta/(2L)$,
\begin{equation} \label{eq:bound-per-term}
	\max_{i\in [d]}|\bm q_i^T\bm S^T\bm S\bm h|\leq\sqrt{\frac{\e}{2d}}.
\end{equation}
Note that \eqref{eq:bound-per-term} implies $\|\bm Q^T\bm S^T\bm S\bm h\|_2^2\leq\frac{\e}{2}$.
Consequently, combining the results we have that there exists an absolute constant $C$, such that if $m\geq CK^4d\log (4dL/\delta)/\e$,
then with probability at least $1-\delta/L$, 
\begin{equation}
	\sigma_{\min}^2 (\bm S\bm Q) \geq \frac{\sqrt{2}}{2} \;\;\;\; \text{and} \;\;\;\; \|\bm Q^T\bm S^T\bm S\bm h\|_2^2\leq\frac{\e}{2},
\end{equation}  
which are the conditions in \eqref{mycond}.
This completes the proof for the sub-Gaussian sketch. 

We next prove the case for the leverage score sampling matrices, and the proof is again based on Lemma \ref{lem:Drineas2011}. 
Note that leverage score sampling can be viewed as a special case of induced measure sampling.
The first condition in \eqref{mycond} is implied by $\|\bm Q^T\bm S^T\bm S\bm Q-\bm I\|_2\leq 1-\frac{\sqrt{2}}{2}$, which, according to  \cite[Lemma A.1]{malik2022FastAlgorithms}, is satisfied with probability at least $1-\delta/2L$ if $m\geq 35 d\log(4dL/\delta)$. 
For the second condition in \eqref{mycond}, the only difference is that one uses Markov's inequality in place of Bernstein's inequality due to the lack of information on the tail of $\bm q_i^T\bm S^T\bm S\bm h$, and the details are omitted.
Under additional assumptions in \eqref{pang}, Markov's inequality can be replaced by Hoeffding's inequality to yield an improved bound \eqref{lop_more}:
\begin{equation}
	\bm q_i^T\bm S^T\bm S\bm h = \frac{1}{m}\sum_{j\in [m]}\langle\sqrt{m}\bm s_j, \bm q_i\rangle\langle\sqrt{m}\bm s_j, \bm h\rangle, \;\;\;\; i\in [d],
\end{equation} 
with each summand $\langle\sqrt{m}\bm s_j, \bm q_i\rangle\langle\sqrt{m}\bm s_j, \bm h\rangle$ centered and bounded as
\begin{equation}
	|\langle\sqrt{m}\bm s_j, \bm q_i\rangle\langle\sqrt{m}\bm s_j, \bm h\rangle|\leq \max_{i\in [d]}\max_{j\in [N]: \ell_j>0}\frac{r| q_{ij} h_j|}{\ell_j} \leq \max_{i\in [d]}\max_{j\in [N]: \ell_j>0}\frac{d| q_{ij} h_j|}{\ell_j}\leq C,
\end{equation}
where $r$ is the rank of $\bm A$. 
By Hoeffding's inequality, for $t>0$, 
\begin{equation}
	\P\left(|\bm q_i^T\bm S^T\bm S\bm h|\leq t\right)\geq 1- 2\exp (-\frac{mt^2}{2C^2}).
\end{equation}
Setting $t = \sqrt{\e/2d}$ and taking a union bound over $i$ yields that, for $m\geq 4C^2d\log (4dL/\delta)/\e$, with probability at least $1-\delta/2L$, $\|\bm Q^T\bm S^T\bm S\bm h\|_2^2\leq\e/2$.

\printbibliography

\end{document}